\documentclass[11pt]{amsart}

\usepackage{amsmath}
\usepackage{amstext}
\usepackage{amssymb}
\usepackage{amsthm}
\usepackage{enumerate}
\usepackage{bbm}
\usepackage{comment}
\usepackage[USenglish]{babel}
\usepackage[utf8]{inputenc}
\usepackage[T1]{fontenc}
\usepackage{xcolor}
\makeatletter
\def\imod#1{\allowbreak\mkern10mu({\operator@font mod}\,\,#1)}
\makeatother

\theoremstyle{plain}
\newtheorem{thm}{Theorem}
\theoremstyle{definition}
\newtheorem{defi}[thm]{Definition}
\newtheorem{cor}[thm]{Corollary}
\newtheorem*{cor*}{Corollary}
\newtheorem*{lem*}{Lemma}
\newtheorem*{thm*}{Theorem}
\newtheorem{lem}[thm]{Lemma}

\newtheorem{prop}[thm]{Proposition}
\newtheorem{remark}[thm]{Remark}

\newtheorem{claim}[thm]{Claim}

\DeclareMathOperator{\cof}{cof}

\DeclareMathOperator{\dom}{dom}

\DeclareMathOperator{\ot}{ot}
\DeclareMathOperator{\id}{id}

\DeclareMathOperator{\GCH}{GCH}

\DeclareMathOperator{\ZFC}{ZFC}

\DeclareMathOperator{\cf}{cf}
\let\split\undefined
\DeclareMathOperator{\split}{split}

\renewcommand{\id}{\mathrm{id}}

\newcommand{\lan}{\langle}

\DeclareMathOperator{\crt}{crit}
\DeclareMathOperator{\Add}{Add}
\DeclareMathOperator{\Col}{Col}

\newcommand{\bbP}{\mathbb{P}}
\newcommand{\bbQ}{\mathbb{Q}}
\newcommand{\bbR}{\mathbb{R}}

\newcommand{\bbS}{\mathbb{S}}

\newcommand{\ka}{\kappa}

\renewcommand{\phi}{\varphi}

\newcommand{\fri}{\mathfrak{i}}

\newcommand{\spec}{\mathfrak{sp}_{\mathfrak i}}

\newcommand{\calA}{\mathcal{A}}
\newcommand{\calB}{\mathcal{B}}
\newcommand{\calC}{\mathcal{C}}

\newcommand{\calF}{\mathcal{F}}
\newcommand{\calG}{\ensuremath{\mathcal{G}}}

\newcommand{\calI}{\ensuremath{\mathcal{I}}}
\newcommand{\calJ}{\ensuremath{\mathcal{J}}}

\newcommand{\calU}{\mathcal{U}}
\newcommand{\calW}{\mathcal{W}}

\newcommand{\p}{\mathcal{P}}

\newcommand{\ra}{\rangle}

\newcommand{\concatB}{\mathbin{\rotatebox[origin=c]{90}{\scalebox{.7}{(\kern1ex)}}}}

\date{}
\begin{document}

\title{Strong independence and its spectrum}


\author{Monroe Eskew}
\address{University of Vienna, Institute for Mathematics, Kolingasse 14-16, 1090 Vienna, Austria}
\email{monroe.eskew@univie.ac.at}

\author{Vera Fischer}
\address{University of Vienna, Institute for Mathematics, Kolingasse 14-16, 1090 Vienna, Austria}
\email{vera.fischer@univie.ac.at}


\thanks{The authors wish to thank the Austrian Science Fund FWF for the generous support through grants number  START Y1012-N35 (Eskew and Fischer), as well as I4039 (Fischer).}

\makeatletter
\@namedef{subjclassname@2020}{%
  \textup{2020} Mathematics Subject Classification}
\makeatother




\begin{abstract} For $\mu, \kappa$ infinite, say $\calA\subseteq [\kappa]^\kappa$ is a $(\mu,\kappa)$-maximal independent family if whenever $\calA_0$ and $\calA_1$ are pairwise disjoint non-empty in $[\calA]^{<\mu}$ then $\bigcap\calA_0\backslash\bigcup\calA_1 \not= \emptyset$, $\calA$ is maximal under inclusion among families with this property, and moreover all such Booelan combinations have size $\kappa$.  We denote by $\spec(\mu,\kappa)$ the set of all cardinalities of such families, and if non-empty, we let $\mathfrak{i}_\mu(\kappa)$ be its minimal element. Thus, $\mathfrak{i}_\mu(\kappa)$ (if defined) is a natural higher analogue of the independence number on $\omega$ for the higher Baire spaces. 
In this paper, we study $\spec(\mu,\kappa)$ for $\mu,\kappa$ uncountable. Among others, we show that:

\begin{enumerate}
\item  The property $\spec(\mu,\kappa)\neq\emptyset$ cannot be decided on the basis of ZFC plus large cardinals.
\item Relative to a measurable, it is consistent that:
\begin{enumerate}
\item $(\exists \kappa{>}\omega) \, \fri_\kappa(\kappa)<2^\kappa$.
\item \label{between} $(\exists \kappa{>}\omega)\,\kappa^+<\mathfrak{i}_{\omega_1}(\kappa)<2^\kappa$.
\end{enumerate}
To the best knowledge of the authors, (\ref{between}) is the first example of a $(\mu,\kappa)$-maximal independent family of size strictly between $\kappa^+$ and  $2^\kappa$, for uncountable $\kappa$.
\item $\spec(\mu,\kappa)$ cannot be quite arbitrary.
\end{enumerate}
\end{abstract}

\maketitle


\section{Introduction}

In this paper, we set to study higher analogues of the classical notion of maximal independent families on $\omega$. Recall that a family $\calA\subseteq [\omega]^\omega$ is independent if for every two pairwise disjoint non-empty subfamilies $\calA_0$, $\calA_1$ of $\calA$ the set $\bigcap\calA_0\backslash \bigcup\calA_1$ is infinite. A maximal independent family is an independent family which is maximal under inclusion.  There have been studies of independent families on uncountable cardinals $\kappa$, but most of those do not consider the notion of maximality (with the exception of~\cite{Kunen} and~\cite{VFDM2}), as the very existence of such maximal independent families on an uncountable $\kappa$ is not guaranteed by the axioms ZFC.  In fact, by an earlier result of Kunen~\cite{Kunen}, the existence of an analogue for countable Booelan combinations
is equiconsistent with the existence of a measurable cardinal. 

The most mild generalization of the notion for an uncountable regular $\kappa$ has been introduced and studied in~\cite{VFDM2}. For a given regular uncountable $\kappa$, a family $\calA\subseteq [\kappa]^\kappa$ is said to be maximal $\kappa$-independent if for every two pairwise disjoint, non-empty, finite subfamilies $\calA_0$, $\calA_1$ of $\calA$ the set $\bigcap\calA_0\backslash \bigcup\calA_1$ is unbounded in $\kappa$ and the family $\calA$ is maximal with respect to this property under inclusion. In contrast to the most general case, the existence of maximal $\kappa$-independent families is guaranteed by the Axiom of Choice. Thus, the minimal cardinality of a maximal $\kappa$-independent family, denoted $\mathfrak{i}(\kappa)$ and referred to as the $\kappa$-independence number, is defined for every regular uncountable $\kappa$. In~\cite{VFDM2} it is shown that $\kappa^+\leq \mathfrak{i}(\kappa)\leq 2^\kappa$, both $\mathfrak{d}(\kappa)$ and $\mathfrak{r}(\kappa)$ are lower bounds of  $\mathfrak{i}(\kappa)$, there is always a maximal $\kappa$-independent family of cardinality $2^\kappa$, and 
that starting from a measurable $\kappa$, one can obtain a generic extension in which $\mathfrak{i}(\kappa)=\kappa^+<2^\kappa$.

On the other hand, the most general higher analogue of independence in the uncountable is given in Definition~\ref{def.ind.fam}. For infinite cardinals $\mu\leq\kappa$, we consider subfamilies $\calA$ of $[\kappa]^\kappa$ with the property that for every two pairwise disjoint, non-empty $\calA_0$, $\calA_1$ in $[\calA]^{<\mu}$, the Boolean combination $\bigcap\calA_0\backslash \bigcup\calA_1$ is nonempty.  If $\calA$ is in addition maximal and every Booelan combination has size $\kappa$, then we say that $\calA$ is a $(\mu,\kappa)$-maximal independent family, or a $(\mu,\kappa)$-mif.  The set of all cardinalities of $(\mu,\kappa)$-mif's is denoted $\spec(\mu,\kappa)$ and its minimal element, if non-empty, is denoted $\mathfrak{i}_\mu(\kappa)$. 

In Section \ref{section.basic}, apart from obtaining necessary conditions
on $\mu$ and $\kappa$ for the existence of a $(\mu,\kappa)$-mif (see Proposition~\ref{ineq2}), building on earlier result of Kunen~\cite{Kunen}
we show in Corollary~\ref{mu.kappa.cof}  that:

\begin{thm*}
If there is a $(\mu,\kappa)$-mif, then $
    \hbox{cf}(\kappa)\geq\mu$.
\end{thm*}
Moreover, we show that any $(\mu,\kappa)$-mif has a restriction which is a $(\mu,\kappa)$-\emph{densely} maximal independent family, or $(\mu,\kappa)$-dmif, an object whose existence has remarkable combinatorial consequences, and we give a sufficient condition for the existence of a $(\mu,\kappa)$-mif (see Lemma~\ref{converse}).

In Section \ref{section.density}, we give a generalization of the notion of a density ideal associated to $(\mu,\kappa)$-independent families, particular cases of which can be found in~\cite{VFDM1, VFDM2}, and define the density completeness number $\lambda_\calF$ associated to a $(\mu,\kappa)$-independent family $\calF$, as the completeness of the density ideal $\calI_\calF$.  We define $\lambda_\mu=\min\{\lambda_\calF:(\exists \kappa)\calF\hbox{ is a } (\mu,\kappa)\hbox{-dmif}\}$.  These cardinals must be rather large.
Our results in particular imply that if $\{\mu_i:i<\theta\}$ is a sequence of pairwise distinct uncountable cardinals with the property that for each $i$ there exists a $(\mu_i,\kappa)$-mif for some $\kappa$, then the $\mu_i$'s are separated by distinct weakly Mahlo cardinals.  

In Section \ref{section.fragile} we show that the property $\spec(\mu,\kappa)\neq\emptyset$ can not be decided by ZFC and large cardinals, addressing a question of \cite{VFDM2}. Indeed, by Theorems~\ref{fragile} and~\ref{indest} we have:

\begin{thm*}\hfill
	\begin{enumerate}
	\item If $\mu$ is regular and $\bbP$ is a nontrivial forcing either of size $<\mu$ or satisfying the $\nu$-c.c.\ for some $\nu<\mu$, then $\bbP$ forces $\spec(\mu,\kappa) = \emptyset$ for all $\kappa$.  
	\item If $\kappa$ is supercompact, then there is a forcing extension in which  for all $\kappa$-directed-closed posets $\bbP$ that force $2^\kappa < \aleph_\eta$, 
	where $\eta = (\kappa^{++})^{L(\p(\kappa))}$, $\bbP$ forces $\spec(\kappa,\kappa) \not= \emptyset$.
	\end{enumerate}
\end{thm*}
 
Foreman's Duality Theorem (see \cite{Foreman}) plays central role in our study of the strong spectrum of independence in Section \ref{section.spectrum}. In particular we obtain (see Theorem~\ref{lowi}):


\begin{thm*}If $\kappa$ is measurable and $\mu<\kappa$ is regular, then there is a forcing extension in which
$\fri_{\mu}(\kappa)$ is a singular cardinal strictly below $2^\kappa$.
Thus, in particular, $$\kappa^+<\mathfrak{i}_{\mu}(\kappa)<2^\kappa.$$ 
\end{thm*}
 
To the best knowledge of the authors, this is the first example of a $(\mu,\kappa)$-mif of cardinality strictly between $\kappa^+$ and $2^\kappa$, for uncountable $\kappa$. The theorem answers a question of~\cite{VFDM2}.

It turns out that the situations in which $\mu<\lambda_\mu$ and in which $\mu=\lambda_\mu$ are quite different.  We study the latter case in Section~\ref{section.Sacks}, using the higher analogue of Sacks forcing as in~\cite{VFDM2} to obtain the consistency of $\fri_\kappa(\kappa)=\kappa^+<2^\kappa$ for a strongly inaccessible $\kappa$. Summarizing these results, we obtain the following information about the spectrum beyond the possible value of $\fri_\mu(\kappa)$:

\begin{thm*} Suppose GCH, $\kappa_0$ is supercompact, and $\kappa_1>\kappa_0$ is measurable. Then there are generic extensions in which:
	\begin{enumerate}
\item $|\spec(\kappa_0,\kappa_0)| \geq 2.$
\item $|\spec(\omega_1,\kappa_1)|\geq 2$.
		\item $\spec(\mu,\kappa_1)\neq\emptyset$ for at least two uncountable $\mu$.
	\end{enumerate}
\end{thm*}

In Section~\ref{section.restrictions}, we study restrictions on the strong spectrum of independence, building on techniques of Gitik and Shelah~\cite{gs2}.  We show that if $\mu$ is uncountable but below the first weakly Mahlo cardinal, then $\fri_\mu(\kappa) \geq \min\{ 2^{\lambda_\mu},\lambda_\mu^{+\omega} \}$ (see Theorem~\ref{specrest}).  This contrasts with the situation in which $\fri_\kappa(\kappa)$ is defined, for there it is possible that $\kappa = \lambda_\kappa$ and $\fri_\kappa(\kappa) = \kappa^+<\kappa^{+n} = 2^\kappa$ for any finite $n \geq 2$.  We conclude the paper with interesting remaining open questions in Section~\ref{section.questions}.

\section{Strong Independence}
\label{section.basic}

\begin{defi}\label{def.ind.fam}	
Let $\mu$ be a cardinal.  A family of sets $\calF\subseteq\p(S)$ is called \emph{$\mu$-free} if for all disjoint $s_0,s_1 \subseteq [\calF]^{<\mu}$, $\bigcap s_0 \setminus \bigcup s_1$ is nonempty.  
\end{defi}

We use the convention in this context that $\bigcap \emptyset = S$.  For a partial function $p : \calF \to 2$ defined at $<\mu$-many points, we let $X_p$ denote the Boolean combination coded by $p$, 
$$X_p = \bigcap \{ X\in\calF : p(X) = 1 \} \setminus \bigcup \{ X\in\calF : p(X) = 0 \}.$$  
Using a bijection between $\calF$ and its cardinality, these Boolean combinations are coded conveniently by conditions in the Cohen forcing $\Add(\mu,|\calF|)$.  If $\omega\leq\mu\leq |\calF|$ and $\calF \subseteq\p(S)$ is a $\mu$-free family, then clearly $|\calF|^{<\mu} \leq 2^{|S|}$.

\begin{defi}\label{def.max.ind} \hfill
\begin{enumerate} 
\item A family $\calF \subseteq \p(S)$ is said to be \emph{$(\mu,\kappa)$-independent} when it is $\mu$-free and for all $p \in \Add(\mu,|\calF|)$, $|X_p| = \kappa$. 
\item A family $\calF \subseteq \p(S)$ is a \emph{$(\mu,\kappa)$-maximal indepedent family}, or \emph{$(\mu,\kappa)$-mif}, if it is $(\mu,\kappa)$-independent and cannot be extended to a strictly larger $\mu$-free family $\calF' \subseteq \p(S)$. 
 \item For cardinals $\mu,\kappa$, let the \emph{$(\mu,\kappa)$-independence spectrum}, or $\mathfrak{sp}_\fri(\mu,\kappa)$, be the set of all cardinalities of $(\mu,\kappa)$-mif's.  
 \end{enumerate} 
 \end{defi}
 
 To avoid some pathologies, we only concern ourselves with the situation in which both $\mu$ and $\kappa$ are infinite and $\mu$ is regular, which we will assume implicitly from here on.  

 \begin{defi} \hfill
 \begin{enumerate}
 \item When $\mathfrak{sp}_\fri(\mu,\kappa)$ is nonempty, we let  $\fri_\mu(\kappa)$ denote its minimal element.   
 \item The \emph{independence number for $\kappa$}, $\fri(\kappa)$, is $\fri_\omega(\kappa)$.  
 \item The \emph{$\sigma$-independence number for $\kappa$}, $\fri_\sigma(\kappa)$, is $\fri_{\omega_1}(\kappa)$.  
 \item The \emph{strong independence number for $\kappa$}, $\fri_s(\kappa)$, is $\fri_\kappa(\kappa)$.
 \end{enumerate}
 \end{defi}

The cardinal invariant $\fri(\ka)$ was introduced in~\cite{VFDM2}  for the special case of $\ka$ regular uncountable.  For every regular uncountable cardinal $\kappa$, the Axiom of Choice implies the existence of an $(\omega,\kappa)$-mif and so $\mathfrak{i}(\kappa)$ is defined, and $\kappa^+\leq \mathfrak{i}(\kappa)\leq 2^\kappa$. On the other hand, as some of the main results of this paper indicate, the existence of a $(\mu,\kappa)$-mif for an infinite $\mu$ is not guaranteed.  We continue with examining some necessary conditions for the existence of $(\mu,\kappa)$-mif's.


\begin{prop}
\label{ineq2} Suppose $\mathfrak{sp}_\fri(\mu,\kappa) \not=\emptyset$. Then $\mu\leq 2^{<\mu} \leq \kappa<\fri_\mu(\kappa)^{<\mu}\leq 2^\kappa$.
\end{prop}
\begin{proof}
First that if $\calF$ is a $(\mu,\kappa)$-mif, then every distinct partial function $p : \calF \to 2$ of size $<\mu$ determines a distinct subset of $\kappa$, so $\fri_\mu(\kappa)^{<\mu}\leq 2^\kappa$.

Let us show $\kappa<\fri_\mu(\kappa)^{<\mu}$.  Suppose $\calF\subseteq\p(\kappa)$ is $(\mu,\kappa)$-independent, and $|\calF|^{<\mu} \leq \kappa$.  By the Disjoint Refinement Lemma (see~\cite{BBPSPV}) we can find a pairwise disjoint collection $\{ Y_p : p \in \Add(\mu,|\calF|)\}$, where each $Y_p$ is contained in $X_p$ and has size $\kappa$.  Then we can partition each $Y_p$ into two disjoint sets of the same cardinality, $Y_p^0,Y_p^1$.  Let $Y^0 = \bigcup \{ Y_p^0 : p \in \Add(\mu,|\calF|) \}$.  Then $\calF \cup \{ Y^0 \}$ is a strictly larger $(\mu,\kappa)$-independent family.

To see $2^{<\mu} \leq \kappa$, let $\calF\subseteq\p(\kappa)$ be a $(\mu,\kappa)$-independent family of size $<\mu$.  The collection $\{ X_p : p \in \,^\calF 2\}$ is a partition of $\kappa$, so $2^{|\calF|}\leq\kappa$.  Since each $X_p$ has size $\kappa$, we can choose a set $Y$ such that for each $p : \calF \to 2$, $X_p \cap Y$ and $X_p \setminus Y$ both have size $\kappa$.  Then $\calF \cup \{ Y \}$ is a strictly larger $(\mu,\kappa)$-independent family, and $\calF$ is not maximal.  Therefore, $2^\alpha \leq\kappa$ for each $\alpha<\mu$.
\end{proof}

Note that if $\mu$ is regular, then the inequality $2^{<\mu} < \fri_\mu(\kappa)^{<\mu}$ implies $\mu<\fri_\mu(\kappa)$.



Suppose $\calF\subseteq\p(S)$ is a $\mu$-free family.  An equivalent way of stating that $\calF$ is maximal is to say that for all $A\subseteq S$, there is $p \in \Add(\mu,|\calF|)$ such that either $X_p \cap A = \emptyset$ or $X_p \subseteq A$.  We say that $\calF$ is \emph{densely maximal} when for all $p \in \Add(\mu,|\calF|)$ and all $A \subseteq X_p$, there is $q \leq p$ such that either $X_q \cap A = \emptyset$ or $X_q \subseteq A$. The notion of dense maximality was studied in~\cite{VFDM1} in the countable case and introduced for higher independent families in~\cite{VFDM2}.  A $(\mu,\kappa)$-mif that is moreover densely maximal will be called a \emph{$(\mu,\kappa)$-densely maximal independent family}, or $(\mu,\kappa)$-dmif.


\begin{lem}[Kunen,~\cite{Kunen}]
\label{lem:ideal}
Suppose there is a maximal $\mu$-free family of size $\theta\geq\mu$.  Then for some $\kappa$, there is a $(\mu,\kappa)$-dmif $\calF \subseteq \p(\kappa)$, and $\theta \in \spec(\mu,\kappa)$.  Furthermore, whenever $\theta \in \spec(\mu,\kappa)$, there is a uniform $\mu$-complete ideal $\calI$ on $\kappa$ and a dense embedding $e : \Add(\mu,\theta) \to \p(\kappa)/\calI$.
\end{lem}

\begin{proof}
Suppose $\calF = \{ Y_\alpha : \alpha < \theta \}$ is a maximal $\mu$-free family.  First we claim that for some $p^* \in \Add(\mu,\theta)$, the collection $\{ X_{p^*} \cap Y_\alpha : \alpha \notin \dom p^* \}$ is a densely maximal $\mu$-free family of subsets of $X_{p^*}$.  If this fails, then for all $p \in \Add(\mu,\theta)$, there is $q \leq p$ and $B_q \subseteq X_q$ such that no $r \leq q$ has the property that either $X_r \cap B_q = \emptyset$ or $X_r \subseteq B_q$.  Let $\calA \subseteq \Add(\mu,\theta)$ be a maximal antichain of such $q$ and let $B = \bigcup_{q \in \calA} B_q$.  By the maximality of $\calF$, there is some $p \in \Add(\mu,\theta)$ such that either $X_p \cap B = \emptyset$ or $X_p \subseteq B$.  There is $r \leq p,q$ for some $q \in \calA$, and by construction, both $X_r \cap B_q$ and $X_r \setminus B_q$ are nonempty.  Noting that $X_r \cap B_q \subseteq X_p \cap B$ and $X_r \setminus B_q= X_r \setminus B \subseteq X_p \setminus B$, we have a contradiction.  Thus let $p^* \in \Add(\mu,\theta)$ be as claimed, and let $\kappa = \min\{ |X_q| : q \leq p^* \}$.  Then for some $X_q$ of size $\kappa$, there is a family of $\theta$-many subsets of $X_q$ that is densely maximal and $(\mu,\kappa)$-independent. 

Now suppose we are given a $(\mu,\kappa)$-dmif $\calF\subseteq\p(\kappa)$.
Let $\calI$ be the collection of $Y \subseteq \kappa$ such that for dense set of $p \in \Add(\mu,\theta)$, $X_p \cap Y =\emptyset$.
 Since the intersection of less that $\mu$-many dense open subsets of $\Add(\mu,\theta)$ is dense, $\calI$ is a $\mu$-complete ideal.
 If $Y \subseteq \kappa$ is $\calI$-positive, there is $p \in \Add(\mu,\theta)$ such that for all $q \leq p$, $X_q \cap Y \not=\emptyset$.  By dense maximality, there is $r \leq p$ such that $X_r \subseteq Y$.  This shows that the map $p \mapsto [X_p]_{\calI}$ is a dense embedding of $\Add(\mu,\theta)$ into $\p(\kappa)/\calI$.  Since every $X_p$ has size $\kappa$, $\calI$ contains $[\kappa]^{<\kappa}$.
 \end{proof}

\begin{cor}\label{mu.kappa.cof}
If $\mathfrak{sp}_\fri(\mu,\kappa) \not=\emptyset$, then $\cf(\kappa) \geq \mu$.
\end{cor}

\begin{proof}
If $\mathfrak{sp}_\fri(\mu,\kappa) \not=\emptyset$, then there is a $\mu$-complete uniform ideal on $\kappa$.
\end{proof}

\section{Density ideals and density completeness numbers}
\label{section.density}

This brings us to the following definition:

\begin{defi}\label{def.density.numbers}\hfill
  \begin{enumerate}
  \item Let $\calF \subseteq \p(\kappa)$ be a $\mu$-free family. Then, the collection  $\calI_{\calF}$  of all $Y \subseteq \kappa$ such that for dense set of $p \in \Add(\mu,|\calF|)$, $X_p \cap Y = \emptyset$ is referred to as the \emph{density ideal} of $\calF$ and $\lambda_{\calF}$ denotes the completeness of $\calI_{\calF}$.  Note that $\lambda_\calF \geq\mu$ by the distributivity of $\Add(\mu,|\calF|)$.
  \item If $\mu$ is such that for some $\kappa$, there exists a $(\mu,\kappa)$-dmif $\calF$, let $\lambda_\mu$ be the minimal value of $\lambda_{\calF}$ for such a family. We refer to $\lambda_\calF$ as the density completeness number of $\calF$ and to $\lambda_\mu$ as the density completeness number at $\mu$. 
 \end{enumerate}
\end{defi}

The notion of a density ideal for independent families on $\omega$ has been introduced in~\cite{VFDM1} and a variant of it for $(\omega,\kappa)$-maximal independent families in~\cite{VFDM2}. In both of those cases, the density ideal plays an important role in obtaining maximal independent families which are indestructible by Sacks forcing (in the former case) and the higher analogue of Sacks forcing (in the latter). 


The following gives a partial converse of Lemma \ref{lem:ideal}.  The key extra assumption is (\ref{cofinalJ}).  

\begin{lem}
\label{converse}
Suppose $\mu,\kappa,\theta$ are cardinals and $\calI$ is a $\mu$-complete ideal on $\kappa$ such that:
\begin{enumerate}
\item $\calI$ contains $[\kappa]^{<\kappa}$.
\item There is a dense embedding $e : \Add(\mu,\theta)\to\p(\kappa)/\calI$.
\item \label{cofinalJ} There is $\calJ \subseteq \calI$ that is $\subseteq$-cofinal in $\calI$, and $|\calJ| \leq \theta$.
\end{enumerate}
Then $\theta \in \mathfrak{sp}_\fri(\mu,\kappa)$.  Futhermore, $\calI = \calI_\calF$ for some $(\mu,\kappa)$-dmif $\calF$ of size $\theta$.
\end{lem}
\begin{proof}
Let $f : \theta \to \calJ$ be a function such that for all $Y \in \calI$, there are $\theta$-many $\alpha$ such that $Y \subseteq f(\alpha)$. For $\alpha<\theta$, let $A_\alpha$ be an equivalence-class representative of $e(\{ \langle \alpha,1 \ra \})$ that is disjoint from $f(\alpha)$.  For all $p \in \Add(\mu,\theta)$, if $Y$ is any equivalence-class representative of $e(p)$, then $Y \subseteq_{\calI} A_\alpha$ if and only if $\alpha \in \dom(p)$ and $p(\alpha) = 1$.  Since $\calI$ is $\mu$-complete, we can subtract away all the errors and get that $e(p) = \left[\bigcap_{p(\alpha) = 1} A_\alpha\setminus \bigcup_{p(\alpha) = 0} A_\alpha \right]_{\calI}$.  Thus $\{ A_\alpha : \alpha < \theta \}$ is $\mu$-free.

Let us show it is maximal.  
Let $B \subseteq \kappa$ be arbitrary.  If $B \in \calI$, then there is $\alpha$ such that $A_\alpha \cap B = \emptyset$.  Otherwise, there is $p \in \Add(\mu,\theta)$ such that $X_p \subseteq_{\calI} B$.  There is $\alpha \in \theta \setminus \dom(p)$ such that $f(\alpha) \supseteq X_p \setminus B$.  Then $X_p \cap A_\alpha \subseteq B$.

Finally we show that $\calI = \calI_\calF$, where $\calF = \{ A_\alpha : \alpha < \theta \}$.  Suppose $B \in \calI$.  Then for all $p \in \Add(\mu,\theta)$, there is some $\alpha \in \theta \setminus \dom(p)$ such that $A_\alpha \cap B = \emptyset$, and so $X_{p ^\frown \lan \alpha,1 \ra} \cap B = \emptyset$.  Thus $B \in \calI_\calF$.  Conversely, suppose $B \in \p(\kappa) \setminus \calI$.  Then there is $p \in \Add(\mu,\theta)$ such that $X_p \subseteq B$.  Thus no $q \leq p$ can satisfy $X_q \cap B = \emptyset$, so $B \notin \calI_\calF$.
\end{proof}

We would like to describe how large a cardinal of the form $\lambda_\mu$ must be.  The following is an elaboration on results of Kunen~\cite{Kunen}, using a different argument.

\begin{lem}
\label{lem:gch}
Suppose $\mu$ is regular and uncountable and $\calI$ is a $\mu$-complete ideal on $\kappa$ such that $\p(\kappa)/\calI$ is forcing-equivalent to $\Add(\mu,\theta)$, where $\theta\not=0$.  Then:
\begin{enumerate}
\item\label{gch} $2^{<\mu} = \mu$.
\item\label{sat} $\calI$ is $\mu^+$-saturated.
\item The completeness of $\calI$ is a cardinal $\lambda\geq\mu$ with the following properties:
\begin{enumerate}
\item\label{muinacc}For all $\alpha<\lambda$ and $\beta<\mu$, $\alpha^\beta<\lambda$.
\item\label{muacc} $2^\mu\geq\lambda$.
\item\label{mahlo} $\lambda$ is weakly $\lambda^+$-Mahlo.
\item\label{wc} $\lambda$ has the tree property, so if $\mu=\lambda$, then $\lambda$ is weakly compact.
\end{enumerate}
\end{enumerate}
\end{lem}

\begin{proof}
Let $\lambda \geq \mu$ be the completeness of $\calI$.  By passing to a positive set if necessary, we may assume that $\kappa$ is the union of $\lambda$-many sets from $\calI$.
Since $\p(\kappa)/\calI$ is equivalent to a countably closed forcing,
it follows that $\calI$ is precipitous (see~\cite{TJKP}).  Let $j : V \to M$ be the generic ultrapower embedding, where $M$ is a transitive subclass of $V[G]$.  The critical point of $j$ is $\lambda$.  Since $G$ introduces no new ${<}\mu$-sequences of ordinals, we must have that $\beta^\alpha < \crt(j)$ for each $\beta<\lambda$ and $\alpha<\mu$.  This shows (\ref{muinacc}), which implies (\ref{gch}) in the case $\mu=\lambda$.

Since $\Add(\mu,\theta)$ is $(2^{<\mu})^+$-c.c.\ and $\calI$ is $2^{<\mu}$-complete, $\calI$ is saturated and thus $M^{\lambda} \cap V[G] \subseteq M$ (see~\cite{Foreman}).  Since $G$ does introduce new subsets of $\mu$, and $\p(\mu)^M = \p(\mu)^{V[G]}$, we must have $2^\mu \geq \crt(j)$, establishing (\ref{sat}) and (\ref{muacc}).

Let us show (\ref{gch}) in the case $\mu<\lambda$.  Since $\Add(\mu,\theta)$ forces that $2^{<\mu} = \mu$, this also holds in $M$ by $\mu$-closure.  By elementarity and the fact that $\mu<\crt(j)$, $2^{<\mu} = \mu$ holds in $V$ as well.  Since $\calI$ is $\mu^+$-saturated, $\lambda$ cannot be a successor cardinal.

To see that $\lambda$ is weakly $\lambda^+$-Mahlo, it suffices to show that there is a normal filter $\calF$ over $\lambda$ that is closed under the Mahlo operation $\mathsf{m} : X \mapsto \{ \alpha \in X : X \cap \alpha$ is stationary$\}$.  Let $\calF = \{ X \subseteq \lambda : 1 \Vdash \lambda \in j(X) \}$.  Clearly $\calF$ is normal, so all $X \in \calF$ are stationary.  Since stationarity in $\lambda$ is preserved by both $\lambda$-c.c.\ and $\lambda$-closed forcing preserves stationarity, $X = j(X) \cap \lambda$ is stationary in $M$, so $\lambda \in j(\mathsf{m}(X))$.

To show the tree property, suppose $T$ is a $\lambda$-tree.   In $V[G]$, we find a cofinal branch of $T$ by looking below a node at level $\lambda$ of $j(T)$.  The forcing to produce the branch is either $\lambda$-closed or has chain condition below $\lambda$.  In neither case can the forcing add branches to $\lambda$-Aronszajn trees.
\end{proof}

\begin{cor}
\label{separation}
Suppose $\mu_0<\mu_1$ are uncountable cardinals and for $i=0,1$, $\calF_i$ is a $(\mu_i,\kappa_i)$-dmif. Then $\lambda_{\calF_0}\leq\mu_1$ and $\lambda_{\calF_0}<\lambda_{\calF_1}$.
Furthermore, if $|\calF_0|<\lambda_{\calF_1}$, then $\lambda_{\calF_0}\leq\kappa_0<\mu_1$.
\end{cor}

\begin{proof}
Let $\lambda_i$ denote $\lambda_{\calF_i}$.  Let us first show that $\lambda_0\leq\mu_1$.  Applying part (\ref{muacc}) of Lemma \ref{lem:gch} to the case $\mu = \mu_0$, we have $2^{\mu_0} \geq \lambda_0$.  If $\mu_1 < \lambda_0$, then by part (\ref{gch}) applied to the case $\mu= \mu_1$, $2^{\mu_0} \leq \mu_1<\lambda_0$, a contradiction.
To show that $\lambda_0<\lambda_1$, let us assume that $\lambda_0 \geq \lambda_1$.  By part (\ref{muinacc}) applied to the case $\mu= \mu_1$, $2^{\mu_0} < \lambda_1\leq\lambda_0$, again a contradiction.

To show the final claim, first note that $\lambda_i\leq\kappa_i$, since $\lambda_i$ is the completeness of an ideal on $\kappa_i$ that extends the ideal of bounded sets.  Suppose towards a contradiction that $|\calF_0|<\lambda_1$ and $\mu_1\leq\kappa_0$.  By Proposition \ref{ineq2}, $\kappa_0<|\calF_0|^{<\mu_0} <\lambda_1$.   Since $\lambda_1^{<\mu_1} = \lambda_1$ by Lemma \ref{lem:gch} and $|\calF_1|^{<\mu_1} > \kappa_1 \geq \lambda_1$ by Proposition \ref{ineq2}, we must have $|\calF_1| > \lambda_1>\kappa_0$.  Let 
$j : V \to M \subseteq V[G]$ be a generic embedding arising from $\calI_{\calF_1}$, where $G$ is $\Add(\mu_1,|\calF_1|)$-generic over $V$.  By elementarity and the fact that $\crt(j) = \lambda_1$, $M$ satisfies that $\calF_0$ is a $(\mu_0,\kappa_0$)-mif.  However, by a basic property of Cohen forcing, $G$ adds a set $Y\subseteq\kappa_0$ of size $\kappa_0$ that splits every $X \in \p(\kappa_0)^V$ of size $\geq\mu_1$.  Thus $\calF_0 \cup \{ Y \}$ is $\mu_0$-independent in $V[G]$.  By the closure of $M$, $Y \in M$, and so $\calF_0$ is not maximal in $M$, a contradiction.
\end{proof}

%

Contrasting somewhat with the above result, we will show later using supercompacts that consistently there is a single $\kappa$ with more than one uncountable $\mu<\kappa$, such that $\mathfrak{sp}_\fri(\mu,\kappa) \not=\emptyset$.  It follows from the above result that the associated density completeness numbers must be different.

We obtain the following property of the density ideals:

\begin{prop}
Suppose $\calF$ is a $(\mu,\kappa)$-mif of size $\theta$.  Then $\calI_{\calF}$ has a $\subseteq$-cofinal subset of size $\leq \theta^\mu$.
\end{prop}
\begin{proof}
If $Y \in \calI_{\calF}$, then there is a maximal antichain $\calA \subseteq \Add(\mu,\theta)$ such that $Y \cap X_p = \emptyset$ for all $p \in \calA$.  Then $X_{\calA} = \bigcup_{p \in A} X_p$ is in the dual filter to $\calI_{\calF}$, and $\kappa \setminus X_{\calA} \supseteq Y$.  By Lemma \ref{lem:gch}, $\Add(\mu,\theta)$ has the $\mu^+$-c.c., so the number of such $X_{\calA}$ is at most $\theta^\mu$.
\end{proof}

\section{Fragility and indestructibility}
\label{section.fragile}

On one side, $2^\kappa\in\mathfrak{sp}_\mathfrak{i}(\omega,\kappa)$ for every regular uncountable $\kappa$, and on the other, by the results in the previous sections, for many natural choices of $\omega<\mu,\kappa$, the spectrum $\mathfrak{sp}_\mathfrak{i}(\omega,\kappa)=\emptyset$. Note that, if $\spec(\mu,\kappa) \not=\emptyset$, then $2^\mu$ and $\kappa$ are at least as large as a weakly Mahlo cardinal.  Furthermore, Corollary \ref{separation} implies that if $\{ (\mu_i,\kappa_i) : i < \theta \}$ is a collection of pairs of uncountable cardinals such that the $\mu_i$ are distinct and $\spec(\mu_i,\kappa_i) \not=\emptyset$ for all $i$, then the $\mu_i$ are separated by a sequence of distinct weakly Mahlo cardinals.
In particular, $\spec(\mu,\kappa)$ is nonempty for at most one uncountable $\mu$ below the first weakly Mahlo and consistently $\spec(\mu,\kappa) = \emptyset$ for all uncountable $\mu,\kappa$. 

In this section we will show that the property $\spec(\mu,\kappa) \neq \emptyset$ for uncountable $\mu$ is rather fragile, as it is destroyed by relatively small forcings (see Theorem~\ref{fragile}). Nevertheless, the property $\spec(\kappa,\kappa) \neq \emptyset$ can be indestructible by a wide class of $\kappa$-directed-closed forcings (see Theorem~\ref{indest}).

 We will employ the notion $\kappa$-approximations due to Hamkins~\cite{Hamkins}. Recall that if $V \subseteq W$ are models of set theory, the pair $(V,W)$ satisfies the $\kappa$-approximation property when for all $X \in V$ and all $Y \subseteq X$ in $W$, if $Y \cap a \in V$ for all $a \in \p_\kappa(X)^V$, then $Y \in V$.  We say that a partial order $\bbP$ has the $\kappa$-approximation property if it forces that the pair $(V,V[G])$ has this property.  

Usuba~\cite{Usuba} defined a forcing to have the \emph{strong $\kappa$-chain condition} when it has the $\kappa$-c.c.\ and does not add branches to $\kappa$-Suslin trees.  He noted that $\bbP$ is strongly $\kappa$-c.c.\ if either $\bbP^2$ is $\kappa$-c.c.\ or $\bbP$ is $\mu$-c.c.\ for some $\mu<\kappa$.  
Improving on results of Hamkins~\cite{Hamkins} and Unger~\cite{Unger}, Usuba showed the following:

\begin{thm}[Usuba]
Suppose $\kappa$ is a regular cardinal, and $\bbP$ is a nontrivial $\kappa$-c.c.\ partial order, and $\dot{\bbQ}$ is a $\bbP$-name for a $\kappa$-strategically-closed forcing.  Then $\bbP * \dot{\bbQ}$ has the $\kappa$-approximation property if and only if $\bbP$ has the strong $\kappa$-c.c.
\end{thm}

Another important tool in our analysis is the Duality Theorem of Foreman (see~\cite{Foreman}), which gives an exact characterization of how the structure of quotient algebras of precipitous ideals is affected by forcing.  If $V \subseteq W$ are models of set theory and $\calI$ is an ideal in $V$, then in $W$, let $\bar\calI$ denote the ideal generated by $\calI$, i.e.\ the set of all $X \subseteq \bigcup \calI$ such that $X \subseteq Y$ for some $Y \in \calI$.  It is easy to see that if $\calI$ is $\kappa$-complete in $V$ and $W$ is a $\kappa$-c.c.\ forcing extension, then $\bar\calI$ is $\kappa$-complete in $W$.

\begin{thm}[Foreman]
\label{foreman}
Suppose $\calI$ is a $\kappa$-complete precipitous ideal on $Z$ and $\bbQ$ is a $\kappa$-c.c.\ forcing.  Then there is a regular embedding $e : \bbQ \to \p(Z)/\calI * j(\bbQ)$ such that whenever $G * \bar H$ is generic for 
$\p(Z)/\calI * j(\bbQ)$ and $H = e^{-1}[\bar H]$, the generic embedding $j_{\dot G}$ lifts to $V[H]$.  Moreover, $\bar\calI$ is precipitous in $V[H]$, and there is a dense embedding
$$d : \bbQ * \p(Z)/\bar\calI \to \calB\left(\p(Z)/\calI * j_{\dot G}(\bbQ)\right).$$
The map $d$ extends $e$ and is given by $d(q,\dot A) =  j_{\dot G}(q) \wedge ||[\id]_{\dot G} \in j_{\dot G}(\dot A)||$.

Furthermore, suppose $H * \bar G \subseteq  \bbQ * \p(Z)/\bar\calI$ and $G * H' \subseteq \p(Z)/\calI * j_{\dot G}(\bbQ)$ are two generics generated from each other by the map $d$.  Let $j_0 : V \to M_0 \cong V^Z/G$ and $j_1 : V[H] \to M_1 \cong V[H]^Z/\bar G$ be the two ultrapower maps.  Then $j_1 \restriction V = j_0$, $j_1(H)=H'$, and $M_1 = M_0[H']$.
\end{thm}

\subsection{Fragility} We will start with establishing our non-existence result:

\begin{thm}\label{fragile}
If $\mu$ is a regular cardinal and $\bbP$ is a nontrivial forcing either of size $<\mu$ or satisfying the $\nu$-c.c.\ for some $\nu<\mu$, then $\bbP$ forces $\spec(\mu,\kappa) = \emptyset$ for all $\kappa$.  
\end{thm}

The theorem follows from the conjunction of two lemmas, each of which gives a stronger conclusion about a special case.

\begin{lem}
\label{strongmucc}
Let $\mu$ be a regular cardinal and $\bbP$ be a nontrivial strongly $\mu$-c.c.\ forcing.
Then $\bbP$ forces that there is no $(\mu,\kappa)$-dmif $\calF$ with $\lambda_{\calF} > \mu$.  
\end{lem}

\begin{proof}
Let $G \subseteq \bbP$ be $\bbP$-generic over $V$.  Suppose that in $V[G]$, there is a $(\mu,\kappa)$-dmif $\calF$ of size $\theta$ with $\lambda_{\calF} =\lambda>\mu$.  By Lemma \ref{lem:ideal}, $\p(\kappa)/\calI_\calF$ is forcing-equivalent to $\Add(\mu,\theta)$.  By Lemma \ref{lem:gch}, $V[G] \models 2^{<\mu} = \mu$.  This must also be true in $V$, since otherwise $\bbP$ would collapse $\mu^+$.


In $V[G]$, let $\calI$ be the density ideal $\calI_\calF$, which has completeness $\lambda$ and is $\mu^+$-saturated.  Let $\dot{\calI}$ be a $\bbP$-name for $\calI$, and in $V$, let $\calJ$ be the ideal $\{ X \subseteq \kappa : 1 \Vdash \check X \in \dot{\calI} \}$.  $\calJ$ is $\lambda$-complete.  Let $\lambda'\geq\lambda$ be the completeness of $\calJ$.  Let $e : \p(\kappa)/\calJ \to \calB(\bbP) * \p(\kappa) / \dot{\calI}$ be the map $[X]_{\calJ} \mapsto (|| \check X \notin \dot{\calI} ||, [\check X]_{\dot{\calI}})$.  It is easy to check that $e$ is order-preserving and antichain-preserving.  Thus $\calJ$ is $\mu^+$-saturated.


Let $\bar{\calJ}$ be the ideal generated by $\calJ$ in $V[G]$.  By the definition of $\calJ$, $\bar\calJ \subseteq \calI$.  Working in $V[G]$, let $\{X_\alpha : \alpha < \delta\leq \mu \}$ be a maximal antichain of $\bar\calJ$-positive sets that are in $\calI$.  Then $X= \bigcup_{\alpha<\delta} X_\alpha \in \calI$ as well.  If $Y \subseteq (\kappa \setminus X)$ is $\bar\calJ$-positive, then $Y$ is $\calI$-positive, so $\calI = \bar\calJ \restriction (\kappa \setminus X)$.  This implies in particular that $\lambda'=\lambda$.  Let $\dot A$ be a $\bbP$-name such that $\Vdash_{\bbP} \dot\calI = \bar\calJ \restriction \dot A$.

Let $d : \bbP * \p(\kappa)/\bar\calJ \to \calB\left(\p(\kappa)/\calJ * j_{\dot H}(\bbP)\right)$ be the map given by Theorem \ref{foreman}.  
Let $(B,\dot q) \leq d(1,[\dot A]_{\bar\calI})$.  We claim that $\p(B)/\calJ$ is a $\mu$-distributive partial order in $V$.  Since $\p(B)/\calJ$ has the $\mu^+$-c.c., it suffices to show that it adds no new functions $f : \nu \to \mu$ for $\nu<\mu$ (see \cite[p.\ 246]{TJ}).

Let $H \subseteq \p(B)/\calJ$ be any generic.  Let $G' \subseteq j(\bbP) \restriction q$ be generic over $V[H]$.  Let $\nu<\mu$ and let $f : \nu \to \mu$ be a function in $V[H*G'] \setminus V$.  It suffices to show that $f \notin V[H]$.  Let $G * \bar H \subseteq \bbP * \p(\dot A)/\bar\calJ$ be the generic filter generated by $d^{-1}$.  Since $\p(\kappa)/\calI = \p(A)/\bar\calJ$ is equivalent to a $\mu$-closed forcing, $f \in V[G]$.  Let $\tau \in V$ be a $\bbP$-name such that $f = \tau^G$, and $\Vdash^V_\bbP \tau \notin V$.  

Let $j_0 : V \to M_0 \cong V^\kappa/H$ and $j_1 : V \to M_1 \cong V[G]^\kappa/\bar H$ be the generic ultrapower embeddings.  Since $\crt(j_0) = \crt(j_1) = \lambda > \mu$, $j_1(f) = f$.  By elementarity, $\Vdash^{M_0}_{j_0(\bbP)} j_0(\tau) \notin M_0$.  
Theorem \ref{foreman} also gives that $j_1 \restriction V = j_0$, $j_1(G) = G'$, and $M_1 = M_0[G']$.  Thus $j_0(\tau)^{G'} = j_1(f) = f$, and $f \notin M_0$.  Since $M_0$ is $\lambda$-closed in $V[H]$, $f \notin V[H]$.


To finish, we have two cases to consider.  In the first case, $\p(B)/\calJ$ is a trivial forcing below some condition.  This is equivalent the existence of some $\calJ$-positive $B' \subseteq B$ such that the dual of $\calJ \cap \p(B')$ is an ultrafilter on $B'$.  But then, Theorem \ref{foreman} tells us that $\bbP$ forces $\p(B')/\bar\calJ \cong j(\bbP)/G$, so $\bar\calJ \restriction B'$ is forced to be $\mu$-saturated.  But this is false, since $\p(A)/\bar\calJ$ is equivalent to $\Add(\mu,\theta)$, which is nowhere $\mu$-c.c.

In the second case, $\p(B)/\calJ$ is a nontrivial forcing. Let $H * G' \subseteq \p(B)/\calJ * j_{\dot H}(\bbP)$ be generic.  Since $\p(B)/\calJ$ is $\mu$-distributive, $H$ itself has the property that $H \cap a \in V$ for all $a \in \p_\mu(\p(B)/\calJ)^V$.   But $V[H*G']$ can also be written as a $\bbP * \dot\Add(\mu,\theta)$-generic extension, and by Usuba's Theorem, $\bbP * \dot\Add(\mu,\theta)$ has the $\mu$-approximation property.  Thus $H \in V$, a contradiction.
\end{proof}

\begin{remark}
The hypothesis of Lemma \ref{strongmucc} is optimal in the following sense.  Suppose $\calF$ is a $(\mu,\kappa)$-dmif with $\lambda_\calF > \mu$.  If there exists a forcing that is $\mu$-c.c.\ but not strongly $\mu$-c.c., then there exists one that preserves the maximality of $\calF$.  This is because any $\mu$-c.c.\ forcing that is not strongly $\mu$-c.c.\ projects to a $\mu$-Suslin tree.  There exists a $\lambda$-complete ideal $\calI$ on some $\kappa$ such that $\p(\kappa)/\calI \cong \Add(\mu,\theta)$, for some $\theta$.   If $T$ is a $\mu$-Suslin tree, then $|T| = \mu< \lambda$.  If $G \subseteq T$ is generic, then $\p(\kappa)^V/\calI$ is dense in $\p(\kappa)^{V[G]}/\bar\calI$.  Since $T$ adds no ${<}\mu$-sequences, $\p(\kappa)/\bar\calI$ is still equivalent to $\Add(\mu,\theta)$ as defined in $V[G]$.
\end{remark}

\begin{lem}
\label{fragilemm}
Suppose $\mu$ is a regular cardinal, $\bbP$ is a nontrivial forcing either of size $<\mu$ or satisfying the $\nu$-c.c.\ for some $\nu<\mu$, and $\Vdash_{\bbP} \dot{\bbQ}$ is $\mu$-strategically-closed.
Then $\bbP * \dot{\bbQ}$ forces that there is no $(\mu,\kappa)$-dmif $\calF$ with $\lambda_{\calF} = \mu$.
In particular, $\bbP$ forces that $\spec(\mu,\mu) = \emptyset$ in any $\mu$-strategically-closed extension.
\end{lem}

The following argument owes a debt to Hamkins~\cite{Hamkins}.

\begin{proof}
First let us show that if $\bbP$ is $\nu$-c.c.\ for some $\nu<\mu$ and $\bbP * \dot\bbQ$ forces that there is a $(\mu,\kappa)$-dmif $\calF$ with $\lambda_{\calF} = \mu$, then in fact $|\bbP| < \mu$.  Assume otherwise.  By Lemma \ref{lem:gch}, $\mu$ is strongly inaccessible.  Let $\eta<\mu$ be such that $\Vdash_\bbP 2^\nu \leq \eta$.  Let $\{ \tau_\alpha : \alpha < \eta \}$ be a complete set of representatives for the $\bbP$-names for subsets of $\nu$.  Let $\theta$ be a sufficiently large cardinal and let $M \prec H_\theta$ contain $\bbP,\mu,\nu,\{\tau_\alpha : \alpha < \eta \}$ and be ${<}\nu$-closed and of size $\leq \eta$.  Then $\bbP_0 = M \cap \bbP$ is a regular suborder of $\bbP$, and $\bbP \cong \bbP_0 * \dot\bbP_1$, where $\dot\bbP_1$ is forced to be $\nu^+$-distributive.  (Using \cite[p.\ 246]{TJ}, the $\nu$-c.c.\ implies that it suffices for $\dot\bbP_1$ to add no subsets of $\nu$.)  Let $G \subseteq \bbP_0$ be generic.  If $\bbP_1$ is a nontrivial forcing, then we can inductively construct a tree of height $\nu^+$ with levels of size $<\nu$, such that every node splits.  The nontriviality of $\bbP_1$ allows us to split at successor stages, the distributivity allows us to continue at limit stages, and the chain condition ensures that the levels never grow too large.  But it is a well-known fact that such trees do possess cofinal branches, and this contradicts the $\nu$-c.c.\ of $\bbP_1$.

Now suppose $G * H \subseteq \bbP * \dot{\bbQ}$ is generic over $V$.  Assume towards a contradiction that in $V[G][H]$, for some $\kappa$, there is an ideal $\calI$ on $\kappa$ with completeness $\mu$ such that $\p(\kappa)/\calI$ is equivalent to $\Add(\mu,\theta)$ for some $\theta$.  By the previous paragraph, we can assume $|\bbP|<\mu$.  By Lemma \ref{lem:gch}, $\mu$ is strongly inaccessible in $V[G][H]$, and so it is as well in $V$ and in any $\mu$-distributive extension of $V[G][H]$.  Let $K$ be $\Add(\mu,\theta)$-generic over $V[G][H]$.

There is an embedding $j : V[G][H] \to M \subseteq V[G][H][K]$, where $\crt(j) = \mu$, and $M$ is a $\mu$-closed transitive class in $V[G][H][K]$.  Let $X \subseteq \mu$ be a set in $V[G][H][K] \backslash V[G]$, which is necessarily in $M$.  Write $M$ as $N[G][H']$.

For a cobounded set of $\alpha<\mu$, $V_\alpha^N = V_\alpha^V$, and thus $V_\mu^N = V_\mu^V$.  Since $j(\bbQ)$ is $j(\mu)$-closed in $N[G]$, $X \in N[G]$.
Let $\dot X \in N$ be a $\bbP$-name for $X$, which we can take so that $\dot X \subseteq V_\mu^N = V_\mu^V$.

For all $a \in \p_\mu(V_\mu)^V$, $a \cap \dot X \in V_\mu^N = V_\mu^V$.  
By Usuba's Theorem, $\bbP * \dot{\bbQ} * \dot{\Add}(\mu,\theta)$ has the $\mu$-approximation property.  Thus $\dot X \in V$.
Therefore, $X = \dot X^G \in V[G]$, which is a contradiction.
\end{proof}

\subsection{Indestructibility}

\begin{thm}
\label{indest}
Suppose $\kappa$ is supercompact.  Then there is a forcing extension in which  for all $\kappa$-directed-closed posets $\bbP$ that force $2^\kappa < \aleph_\eta$, 
where $\eta = (\kappa^{++})^{L(\p(\kappa))}$, $\bbP$ forces $2^\kappa \in \spec(\kappa,\kappa)$.
\end{thm}

\begin{remark} The class of posets in the hypothesis of Theorem~\ref{indest} is cofinal in the class of $\kappa$-directed-closed forcings, since every $\kappa$-directed-closed $\bbP$ is a complete suborder of $\bbP * \dot{\Add}(\kappa^+,1)$, which forces $2^\kappa = \kappa^+$.
\end{remark}

\begin{proof}
We may assume that GCH holds in $V$.  Let $f : \kappa \to V_\kappa$ be
a Laver function (see~\cite[Theorem 20.21]{TJ}).  Inductively define the following Easton-support iteration $\lan \bbP_\alpha,\dot{\bbQ}_\alpha : \alpha < \kappa \ra$.  Suppose we have defined the iteration up to $\alpha$.  If $\alpha$ is inaccessible in $V$, $\bbP_\alpha$ is $\alpha$-c.c., and $f(\alpha)$ is a $\bbP_\alpha$-name for an $\alpha$-directed-closed poset of size $\beta_\alpha$, let $\dot{\bbQ}_\alpha$ be a $\bbP_\alpha$-name for $f(\alpha) \times \Add(\alpha,\beta_\alpha)$.  Otherwise, let $\dot{\bbQ}_\alpha$ be trivial.

Let $G \subseteq \bbP_\kappa$ be generic, and let $\bbQ$ be a $\kappa$-directed-closed poset in $V[G]$ that forces $2^\kappa < \aleph_\eta$, where $\eta = (\kappa^{++})^{L(\p(\kappa))}$.  Let $\lambda$ be the least regular cardinal $\geq \kappa^+ + |\bbQ|$.  By the properties of the Laver function, there is a normal ultrafilter $\calU$ on $\p_\kappa(\lambda)$ such that $j_{\calU}(f)(\kappa) = \dot{\bbQ}$.  Force over $V[G]$ to get a generic $H_0 \times H_1 \subseteq \bbQ \times \Add(\kappa,\lambda)$.  In $M = \mathrm{Ult}(V,\calU)$, the quotient forcing $j(\bbP_\kappa)/ (G * H_0 *H_1)$ is $\lambda^+$-closed.  By the GCH, we can build a generic $K$ for this poset in $V[G*H_0*H_1]$.  Thus we can lift the embedding to $j : V[G] \to M[G']$, where $G' = G*H_0*H_1*K$.  Since $\bbQ$ is $\kappa$-directed-closed and $j(\kappa)>\lambda$, there is a condition $q$ below $j[H_0]$ in $j(\bbQ)$.  By the GCH again, $|j(\lambda)| = \lambda^+$, so we can build a generic $H_0'$ for this poset below $q$ and extend the embedding further to $j : V[G*H_0] \to M[G'*H_0']$.

Over the model $V[G*H_0]$, we only need to force with $\Add(\kappa,\lambda)$ to obtain the lifted embedding as above.  In $V[G*H_0]$, define a normal ideal on $\kappa$ as $\calI = \{X \subseteq \kappa : 1 \Vdash_{\Add(\kappa,\lambda)} \kappa \notin j(X) \}$.  Since $\Add(\kappa,\lambda)$ is $\kappa^+$-c.c., the ideal is seen to be $\kappa^+$-saturated by mapping an $\calI$-positive set $X$ to $|| \kappa \in j(X) ||$.  By normality, this embedding is regular, since if $\lan X_\alpha : \alpha < \kappa \ra$ is a maximal antichain, then $\nabla_\alpha X_\alpha$ is $\calI$-equivalent to the top condition, so it is forced that $\kappa \in j(\nabla_\alpha X_\alpha)$, which means that one of the Boolean values $|| \kappa \in j(X_\alpha) ||$ is in the generic filter.  Thus if we force with $\p(\kappa)/\calI$, a further forcing allows us to obtain the extension of $j$ as above.  If $i : V \to N$ is a generic ultrapower embedding arising from $\calI$, then there is a factor map $k : N \to M$ such that $j = k \circ i$ defined by $k([g]) = j(g)(\kappa)$. This holds because for $X \subseteq \kappa$, $X$ is in the projected generic ultrafilter $\calW \subseteq \p(\kappa)/\calI$ iff $\kappa \in j(X)$ by the definition of the Boolean embedding.

We claim that $\p(\kappa)/\calI$ is actually equivalent to $\Add(\kappa,\lambda)$.  This will complete the proof, since $2^\kappa = \lambda$ in $V[G*H_0]$, so we apply Lemma \ref{converse} to obtain a $(\kappa,\kappa)$-dmif.  To show the claim, it suffices to show that if $H_1 \subseteq \Add(\kappa,\lambda)$ is generic and $\calW$ is the projected generic ultrafilter for $\p(\kappa)/\calI$, then $H_1$ can be recovered from $\calW$.

Let $k : N \to M$ be the factor embedding as above.  Since $\p(\kappa)^{V[G*H_0]} \subseteq N$, all the relevant models agree on $\kappa^+$.  Thus the critical point of $k$ is greater than $\kappa^+$.  Also, $(2^\kappa)^N \geq \lambda$, and $k((2^\kappa)^N) = (2^\kappa)^{M[G'*H_0']} = \lambda$, so $(2^\kappa)^N = \lambda$ and $k(\lambda)=\lambda$.  
Let $\calC$ be the set of cardinals of $V[G*H_0]$ that are $\leq\lambda$.  By hypothesis, $\calC = \calC^{M[G'*H_0']}$ is a set of order-type $<\crt(k)$.
Since $k(\calC^N) = \calC$, we must have that the image $k[\calC^N] = \calC$.  Thus the critical point of $k$ is $>\lambda$.  Now, $G$ is generic for an iteration as defined above, and in $N$ we can examine $i(G)_\kappa$.  This is a filter $h_0 \times h_1$, and $k(h_0 \times h_1) = H_0 \times H_1$.  Since $H_1$ is of rank $<\crt(k)$, we have $h_1=H_1$.  This shows that $H_1$ can be definably recovered from $\calW$, completing the argument.
\end{proof}

\begin{cor}
If there is a supercompact $\kappa$, then for every successor ordinal $\xi < \aleph_\eta$, where $\eta = (\kappa^{++})^{L(\p(\kappa))}$, there is a generic extension in which $\spec(\kappa,\kappa) = \{ \kappa^{+\xi} \}$.
\end{cor}

\begin{proof}
First note that $(\kappa^{++})^{L(\p(\kappa))}$ can only increase as we pass to outer models.  Let $\xi$ be as hypothesized, and if necessary, first pass to an extension in which $\kappa$ is supercompact and $\GCH$ holds.  Now suppose $G \subseteq \bbP_\kappa$ is generic for the iteration defined in the proof of Theorem \ref{indest}.  Then $\GCH$ holds above $\kappa$ in $V[G]$.  Let $H \subseteq \Add(\kappa,\kappa^{+\xi})$ be generic over $V[G]$.  $V[G][H] \models 2^\kappa = \kappa^{+\xi}$, so Theorem \ref{indest} implies that in $V[G][H]$, $\kappa^{+\xi} \in \spec(\kappa,\kappa)$. To see that no smaller cardinals are in $\spec(\kappa,\kappa)$, note that any $\kappa$-free family $\calF \subseteq \p(\kappa)$ occurs in $V[G][H{\restriction}X]$, where $X \subseteq \kappa^{+\xi}$ has size at most $|\calF|^\kappa$.  Thus if $|\calF| < \kappa^{+\xi}$, then there is a $\kappa$-Cohen generic set $Y$ over $V[G][H{\restriction}X]$ in $V[G][H]$, and $Y$ splits every unbounded subset of $\kappa$ in $V[G][H{\restriction}X]$.  Thus $\calF$ is not maximal in $V[G][H]$.
\end{proof}

\section{The spectrum of strong independence}
\label{section.spectrum}

Lemma \ref{lem:ideal} shows that if we want a model where for some uncountable regular cardinal $\mu$, $\mathfrak{sp}_\fri(\mu,\kappa)\not=\emptyset$ for some $\kappa$, there must exist a countably complete precipitous ideal on $\kappa$ whose quotient algebra has a very particular form.  This suggests the relevance of Theorem \ref{foreman} to constructing models where such families exist.  Indeed, it allows us to give a rather straightforward proof of a result of Kunen (see~\cite{Kunen}).



\begin{thm}[Kunen]
It is consistent relative to a measurable cardinal that $\spec(\omega_1,\kappa) \not=\emptyset$ for some $\kappa$. 
\end{thm}
\begin{proof}
Let $\calU$ be a $\kappa$-complete ultrafilter on $\kappa$ and let $\mu<\kappa$ be regular.  Let $H \subseteq \Add(\mu,\kappa)$ be generic over $V$.  By Foreman's Theorem, the ideal $\calJ$ generated by the dual of $\calU$ in $V[H]$ has the property that $\p(\kappa) / \calJ \cong \Add(\mu,2^\kappa)$, since in $V$, $|j(\kappa)| = 2^\kappa$.  By Lemma \ref{converse}, there is a $(\mu,\kappa)$-dmif $\calF \subseteq \p(\kappa)$ in $V[H]$.
\end{proof}

For uncountable cardinals $\mu$, we have thus far only seen examples where either $\spec(\mu,\kappa) = \emptyset$ or $2^\kappa \in \spec(\mu,\kappa)$.
An observation of Gitik and Shelah \cite{gs2} allows us to find a model where for some uncountable $\mu<\kappa$, $\fri_\sigma(\kappa)$ is defined and takes a value strictly less than $2^\kappa$, thus in particular addressing a question from~\cite{VFDM2}.

\begin{thm}
\label{lowi}
It is consistent relative to a measurable cardinal that for some $\kappa$, $\fri_\sigma(\kappa)$ is a singular cardinal below $2^\kappa$.  If there is a supercompact cardinal $\lambda$, then for any ordinals $\alpha,\beta$ with $\beta<\lambda$, there is a forcing extension in which for some $\kappa<\lambda$, $\fri_\sigma(\kappa) = \lambda^{+\beta}$ and $2^\kappa \geq \lambda^{+\alpha}$.
\end{thm}

\begin{proof}
For the first claim, assume $\kappa$ is a measurable cardinal such that $2^\kappa<\kappa^{+\omega}$.  Let $\calU$ be a $\kappa$-complete ultrafilter on $\kappa$ and let $\mu<\kappa$ be regular.  Since $2^\kappa<\kappa^{+\omega}$, $j_{\calU}(\kappa^{+n})=\kappa^{+n}$ for large enough $n<\omega$, and thus $j_{\calU}(\kappa^{+\omega}) = \kappa^{+\omega}$.  Let $\bbQ = \Add(\mu,\kappa^{+\omega})$.  The map $e$ given by Foreman's Theorem is essentially an embedding of $\bbQ$ into itself given by moving coordinates by $j_{\calU}$.  Thus if $H \subseteq \bbQ$ is generic, then if $\calJ$ is the ideal generated by the dual of $\calU$ in $V[H]$, we have $\p(\kappa)/\calJ \cong \Add(\mu,\kappa^{+\omega})$.  Since $|\calU| < \kappa^{+\omega}$, Lemma \ref{converse} implies that in $V[H]$, $\kappa^{+\omega} \in \mathfrak{sp}_\fri(\mu,\kappa)$. 

To show that $\kappa^{+\omega}$ is the minimum element of $\mathfrak{sp}_\fri(\mu,\kappa)$ in $V[H]$, suppose $\calF \subseteq \p(\kappa)$ is a $\mu$-independent family of size $<\kappa^{+\omega}$.  By the chain condition of $\bbQ$, there is some regular $\bbP \subseteq \bbQ$ of size $<\kappa^{+\omega}$ such that $\calF \in V[H \cap \bbP]$.  We can take $\bbP$ to be a sub-product of $\bbQ$, i.e.\ the collection of all partial binary functions on $A$ of size $<\mu$, where $A \subseteq \kappa^{+\omega}$ has size $<\kappa^{+\omega}$.  Let $\calG$ be the set of all ${<}\mu$-length Boolean combinations of members of $\calF$, and let $B \subseteq \kappa^{+\omega}$ be a set of $\kappa$-many ordinals not in $A$, $\lan\xi_\alpha : \alpha<\kappa \ra$.  In $V[H]$, let $Y = \{ \alpha<\kappa : H(\xi_\alpha) = 1 \}$.  For every $X \in \calG$, both $X$ and $\kappa\setminus X$ have size $\kappa$.  Thus by genericity, $X \cap Y$ and $X \setminus Y$ has size $\kappa$ for every $X \in \calG$.  Therefore, $\calF$ is not maximal and $\kappa^{+\omega} = \fri_\mu(\kappa)$.  Since $\cf(2^\kappa)>\kappa$, $2^\kappa > \fri_\mu(\kappa)$.

For the second claim, suppose $\lambda$ is supercompact $\beta<\lambda$, and $\alpha$ is any ordinal.  First pass to a forcing extension in which $\lambda$ is still supercompact and $2^\lambda \geq \lambda^{+\alpha}$.  Then use Priky forcing to make $\lambda$ a singular strong limit cardinal of cofinality $\omega$.  Let $\beta<\kappa<\lambda$ be such that $\kappa$ is measurable with witness $\calU$.  Again, $\lambda$ is a fixed point of $j_{\calU}$, and so are the first $\kappa$ cardinals above $\lambda$.  Now force with $\Add(\mu,\lambda^{+\beta})$, where $\mu<\kappa$ is regular.  By the same arguments as above, $\fri_\mu(\kappa) = \lambda^{+\beta}$ in this model.  We have $2^\kappa = (2^\kappa)^{\kappa} \geq \lambda^\omega \geq\lambda^{+\alpha}$.
\end{proof}

For the models constructed above where $\fri_\mu(\kappa)<2^\kappa$, we do not know whether the set $\spec(\mu,\kappa)$ contains more than one cardinal.  Thus it is of interest to ask if consistently $|\spec(\mu,\kappa)|\geq 2$. 
\begin{thm}
\label{twoi}
Let $\mu<\kappa_0<\kappa_1$ be regular cardinals such that $\kappa_0$ is strongly compact and $\kappa_1$ is measurable.  Then there is a generic extension in which $\spec(\mu,\kappa_1)$ contains at least two cardinals.
\end{thm}

\begin{proof}
Let $\calU_0$ be a fine $\kappa_0$-complete ultrafilter on $\p_{\kappa_0}(\kappa_1)$ and let $\calU_1$ be a $\kappa_1$-complete ultrafilter on $\kappa_1$.  Let $j_i : V \to M_i$ be the ultrapower embedding associated to $\calU_i$.  Let $\lambda > \kappa_1$ be a strong limit cardinal of cofinality $\kappa_0$.  There is an increasing sequence of $\lan \lambda_i : i < \kappa_0 \ra$ cofinal in $\lambda$, such that each $\lambda_i$ is a fixed point of both $j_0$ and $j_1$.  Clearly, $\lambda$ is a fixed point of $j_1$ but not of $j_0$.
\begin{claim}
If $\lambda \leq \lambda' \leq 2^\lambda$, then $| j_0(\lambda') | = 2^\lambda$.
\end{claim}
\begin{proof}
Let $\theta=2^\lambda$.  Let $\lan f_\alpha : \alpha < \theta \ra$ be a sequence of distinct members of $\prod_{i<\kappa_0} \lambda_i$.  Since each $\lambda_i$ is fixed by $j_0$, for each $\alpha<\theta$, $j_0(f_\alpha) \restriction \kappa_0 \in \prod_i \lambda_i$.  If $\alpha<\beta$, then there is $i$ such that $f_\alpha(i) \not= f_\beta(i)$, so $j_0(f_\alpha)(i) \not= j_0(f_\beta)(i)$.  Thus $M_0 \models 2^\lambda \geq \theta$.   Since $M_0 \models j_0(\lambda)$ is a strong limit, $\theta < j_0(\lambda)$.  On the other hand, if $\lambda' \leq 2^\lambda$, then $|j_0(\lambda')| \leq (2^\lambda)^{\kappa_1}= 2^\lambda$.
\end{proof}

Let $\alpha<\kappa_1$ be such that $\lambda^{+\alpha} < 2^\lambda$ and let $H \subseteq \Add(\mu,\lambda^{+\alpha})$ be generic.  By the same argument as for Theorem \ref{lowi}, if $\calI_i \in V[H]$ is the ideal generated by the dual to $\calU_i$, then $\p(\p_{\kappa_0}(\kappa_1)^V) / \calI_0 \cong \Add(\mu,2^\lambda)$ and $\p(\kappa_1) / \calI_1 \cong \Add(\mu,\lambda^{+\alpha})$.  Since $\kappa_1$ is inaccessible in $V$, all $\calI_0$-positive sets have cardinality $\kappa_1$, and  since $|\calU_0| = |\calU_1| = (2^{\kappa_1})^V < \lambda$, by Lemma \ref{converse}, $\{\lambda^{+\alpha},2^\lambda\} \subseteq \spec(\mu,\kappa_1)$.
\end{proof}

%

\begin{thm}\label{at.least.two}
Assume $\GCH$, $\kappa_0$ is supercompact and $\kappa_1>\kappa_0$ is measurable.  Then there is a generic extension in which $\spec(\mu,\kappa_1) \not=\emptyset$ for at least two uncountable $\mu$.
\end{thm}

\begin{proof}
Let $f : \kappa_0 \to V_{\kappa_0}$ be a Laver function, so that for all $\lambda$ and all $X \subseteq V_\lambda$, there is a normal $\kappa_0$-complete ultrafilter $\calU$ on $\p_{\kappa_0}(\lambda)$, such that $j_{\calU}(f)(\kappa_0) = X$.  Let $\bbP_{\kappa_0}$ be an Easton-support iteration contained in $V_{\kappa_0}$, where at stage $\alpha$, we force with the trivial poset, unless $\alpha$ is inaccessible, $f\restriction \alpha \subseteq V_\alpha$, and $f(\alpha)$ is an ordinal, in which case we force with $\Add(\alpha,f(\alpha)^+) \times \Add(\alpha^+,f(\alpha))$.  Let $\calU$ be a normal $\kappa_0$-complete ultrafilter on $\p_{\kappa_0}(\kappa_1)$ such that $j(f)(\kappa_0) = \kappa_1$, where $j : V \to M$ is the ultrapower embedding. 

Let $G \subseteq \bbP_{\kappa_0}$ be generic.  Then let $H_1 \subseteq \Add(\kappa_0^+,\kappa_1)$ be generic over $V[G]$.  We can lift $j$ to an embedding with domain $V[G][H_1]$ via the following procedure.  First force to get a generic $H_0 \subseteq \Add(\kappa_0,\kappa_1^+)$.  By the closure of $M$ and the chain condition of the forcing, $M[G][H_0][H_1]$ is a $\kappa_1$-closed subclass of $V[G][H_0][H_1]$.  The remaining tail-end of $j(\bbP_{\kappa_0})$ is $\kappa_1^+$-closed in $M[G][H_0][H_1]$.  By the $\GCH$, $|j(\kappa_1) | = \kappa_1^+$, and $\kappa_1^{++}$ is a fixed point of $j$.  Thus we can build a filter $K$ that is generic over $M[G][H_0][H_1]$ for the remaining part of $j(\bbP_{\kappa_0})$ and lift to $j : V[G] \to M[G']$, where $G'= G *H_0*H_1*K$.  To lift further to include $H_1$ in the domain, note that by normality, $j \restriction H_1 \in M[G][H_1]$.  Since $|H_1| = \kappa_1$, there is a lower bound to $j[H_1]$ in $\Add(j(\kappa_0^+),j(\kappa_1))^{M[G']}$.  Since $|j(\kappa_1)| =\kappa_1^+$, we can again build a filter $H_1'$ for this forcing that is generic over $M[G']$, with $j[H_1] \subseteq H_1'$.  

Thus we can lift the map to $V[G][H_1]$ by adding the object $H_0$.  Let $\calI$ be the normal ideal on $Z=\p_{\kappa_0}(\kappa_1)^V$ defined as $\{ X : 1 \Vdash_{\Add(\kappa_0,\kappa_1^+)} j[\kappa_1] \notin j(X) \}$.  As in the argument for Theorem \ref{indest}, the map $[X]_{\calI} \mapsto || j[\kappa_1] \in j(X) ||$ is a regular embedding of $\p(Z)/\calI$ into the Boolean completion of $\Add(\kappa_0,\kappa_1^+)$.  If $\calU'$ is a generic ultrafilter projected from $H_0$ and $i : V[G][H_1] \to N$ is the generic ultrapower embedding, then there is a factor map $k : N \to M[G'][H_1']$.  Since $\calI$ is normal, for every ordinal $\alpha\leq\kappa_1$, $\alpha = \ot(i(\alpha) \cap [\id]_{\calU'})= \ot(j(\alpha) \cap j[\kappa_1]) = k(\alpha)$, so $\crt(k) > \kappa_1^+$.  Thus, if $h_1$ is the first factor of $i(G)_\kappa$, then $k(h_0) = H_0$ and $h_0 = H_0$, so $H_0$ can be recovered from $\calU'$.  Therefore, $\p(Z)/\calI$ is forcing-equivalent to $\Add(\kappa_0,\kappa_1^+)$.  Since $|\calI| = \kappa_1^+$, Lemma \ref{converse} implies that $\kappa_1^+ \in \spec(\kappa_0,\kappa_1)$.

Now, $\kappa_1$ is still measurable in $V[G]$.  Let $\calW$ be a $\kappa_1$-complete ultrafilter on $\kappa_1$ in $V[G]$.  If $\calJ$ is the ideal generated by the dual of $\calW$ in $V[G][H_1]$, then using Theorem \ref{foreman}, $\p(\kappa_1)/\calJ$ is forcing-equivalent to $\Add(\kappa_0^+,j_{\calW}(\kappa_1))$.  By $\GCH$ and Lemma \ref{converse}, $\kappa_1^+ \in \spec(\kappa_0^+,\kappa_1)$.
\end{proof}

\section{Sacks indestructibility}\label{section.Sacks}

In~\cite{VFDM2}, the authors show that it is consistent relative to a measurable that $\mathfrak{i}(\kappa)=\mathfrak{i}_\omega(\kappa)<2^\kappa$ by constructing a special $(\omega,\kappa)$-dmif which is indestructible by large products of $\kappa$-Sacks forcing. Thus our Theorem~\ref{strong_indestructible} is a generalization of this result to $(\kappa,\kappa)$-dmif's for $\kappa>\omega$.

Below, $\mathbb{S}_\kappa$ denotes $\kappa$-Sacks forcing and $\bbS_\kappa^\lambda$ the $\kappa$-supported product of $\lambda$-many copies of $\kappa$-Sacks forcing.  The $\kappa$-Sacks forcing was introduced by Kanamori \cite{Kanamori}, where some of its key properties are established.  The conditions are perfect trees $T \subseteq \! ^{<\kappa} 2$ satisfying certain conditions, ordered by $T \leq S$ when $T \subseteq S$.   Among others, Kanamori proved \cite[Lemma 1.2]{Kanamori} that if $\beta<\kappa$ and $\lan p_\alpha : \alpha < \beta \ra$ is a decreasing sequence of conditions, then $\bigcap_{\alpha<\beta} p_\alpha$ is a condition.  Thus $\bbS_\kappa$ is $\kappa$-closed with greatest lower bounds.  This makes Theorem \ref{indest} applicable, because of the following folk result mentioned by Velleman \cite{Velleman}:

\begin{prop}
If $\kappa$ is a regular uncountable cardinal and
 $\bbP$ is a $\kappa$-closed poset with greatest lower bounds, then $\bbP$ is $\kappa$-directed-closed.
\end{prop}

\begin{proof}
If $D \subseteq \bbP$ is a countable directed set, then we can easily find a sequence $p_0 \geq p_1 \geq p_2 \geq \dots$ contained in $D$ such that for all $q \in D$, there is $n$ such that $p_n \leq q$.  Thus $D$ has a greatest lower bound.  Let $\mu<\kappa$ be a cardinal, and suppose by induction that for every directed $D \subseteq \bbP$ of cardinality $<\mu$, $D$ has a greatest lower bound.  Let $D$ be a directed set of size $\mu$, and let $\lan D_\alpha : \alpha <\mu\ra$ be a $\subseteq$-increasing sequence such that each $D_\alpha$ is a directed set of size $<\mu$, and $\bigcup_{\alpha<\mu} D_\alpha = D$.  If $p_\alpha = \inf D_\alpha$, then $\lan p_\alpha : \alpha < \mu \ra$ is a decreasing sequence, so by hypothesis it has a greatest lower bound.
\end{proof}

For an ideal $\calI$, let $\calI^*$ denote its dual filter.

\begin{lem}\label{lem.Sacks.ind} Let $\calF$ be a $(\kappa,\kappa)$-dmif such that its density ideal $\calI_\calF$ is normal. Thus, in particular:
	\begin{enumerate}
		\item Every subfamily of $\calI_\calF^*$ of cardinality $\leq\kappa$ has a pseudo-intersection in $\calI_\calF^*$.
		\item For every strictly increasing function $f\in{^\kappa\kappa}$ there is $C\in\calI^*_\calF$ such that if $C=\{k(\alpha)\}_{\alpha\in\kappa}$ is the increasing enumeration of $C$, then $f(k(\alpha))<k(\alpha+1)$.
	\end{enumerate}
Then $\calF$ remains maximal in generic extensions obtained via arbitrarily large products $\bbS^\lambda_\kappa$ of $\kappa$-Sacks forcing $\bbS_\kappa$.
\end{lem}
\begin{proof}
Since the proof below is a straightforward generalization of the proof appearing in~\cite[Theorems 43 and 45]{VFDM2} we give only an outline in the case $\lambda=1$. Note that a $(\kappa,\kappa)$-independent family $\calF$ is densely maximal iff for every $p\in\Add(\kappa,|\calF|)$ and every $A\subseteq X_p$ either $X_p\backslash A\in\calI_\calF$ or there is $q\leq p$ such that $X_q\subseteq X_p\backslash A$ (see~\cite[Lemma 27]{VFDM2}). Thus, it is sufficient to show that in $V^{\bbS_\kappa}$ for every $A\subseteq \kappa$ and every $p\in\Add(\kappa,|\calF|)$ such that $A\subseteq X_p$ the following holds:
$$\hbox{Either }X_p\backslash A\in\calI_\calF\hbox{ or }\exists q\leq p\hbox{ such that }X_q\subseteq X_p\backslash A.$$
Suppose this is not the case.  Since $\bbS_\kappa$ is $\kappa$-closed, it does not add ${<\kappa}$-sequences of members of $\calF$. Thus we can find $r\in\bbS_\kappa$, an $\bbS_\kappa$-name $\dot{A}$ and $p\in\Add(\kappa,|\calF|)^V$ such that 
$$r\Vdash_{\bbS_\kappa}\dot{A}\subseteq X_p\land X_p\backslash \dot{A}\notin\calI_\calF\land\forall q\leq p(X_q\cap\dot{A}\neq\emptyset).$$
Without loss of generality $r\in\bbS_\kappa$ is preprocessed for $\dot{A}$ (see~\cite[Definition 35 and Lemma 37]{VFDM2}), i.e. for all $t\in\split_\alpha(r)$ there is $x_t\in{^\alpha 2}$ such that $r_t\Vdash_{\bbS_\kappa}\dot{A}\upharpoonright\alpha=\check{x}_t$. For each $t\in\split(r)$ let $$Y_t=\{\beta\in\kappa:\exists t'\in\split(r)\hbox{ s.t. }t'\hbox{ is comparable with }t\hbox{ and }r_{t'}\Vdash\check{\beta}\in\dot{A}\},$$ i.e. $Y_t$ is the outer hull of $\dot{A}$ below $r_t$ (see~\cite[Definition 38]{VFDM2}). Then $Y_t\subseteq X_p$ and for all $t\in\split(r)$ $Y_t\cup\kappa\backslash X_p\in\calI_\calF^*$. This is because, by dense maximality, either $X_p \setminus Y_t \in \calI_\calF$, or there is $q \leq p$ such that $X_q \subseteq X_p \setminus Y_t$.  But the latter cannot occur, since $r_t \Vdash \emptyset \not= X_q \cap \dot A \subseteq X_q \cap Y_t$.

By property $(1)$ of $\calI^*_\calF$, there is $C\in\calI_\calF^*$ such that $C \subseteq^* Y_t\cup\kappa\backslash X_p$  for each $t\in\split(p)$ (where $D \subseteq^* E$ means $D \setminus E$ is bounded in $\kappa$).
Then $C\cap X_p\subseteq^* Y_t$ for all $t\in\split(r)$ and so there is $f\in{^\kappa\kappa}$ such that for all $\alpha<\kappa$
$$(C\cap X_p)\backslash f(\alpha)\subseteq\bigcap_{t\in\split_{\alpha+1}(r)}Y_t.$$
Without loss of generality, $f$ is strictly increasing. By property $(2)$ of $\calI_\calF^*$ there is $C^*\in\calI_\calF^*$ such that for all $\alpha\in C^*$ and all $\gamma\in\alpha\cap C^*$, $f^2(\gamma)<\alpha$. Now, let $C'=C\cap C^*\cap (f(1),\kappa)$. Thus, $C'\in\calI_\calF^*$. Let $\{k(\alpha)\}_{\alpha\in\kappa}$ be the increasing enumeration of $C'\cap X_p$. Recursively construct a fusion sequence $\tau=\langle r_\alpha:\alpha<\kappa\rangle$ below $r$ such that
$$\forall \alpha<\kappa, r_\alpha\Vdash k(\alpha)\in\dot{A}.$$
Then the fusion $r^*$ of $\tau$ will force $``C'\cap X_p\subseteq\dot{A}"$ and so 
$$r^*\Vdash_{\bbS_\kappa}X_p\backslash\dot{A}\subseteq X_p\backslash C'.$$
However $X_p\backslash C'\subseteq \kappa\backslash C'\in \calI_\calF$, so we obtain $r^*\Vdash_{\bbS_\kappa} X_p\backslash \dot{A}\in\calI_\calF$, which is a contradiction to the choice of $r$.
\end{proof}

\begin{thm}\label{strong_indestructible}
It is consistent relative to a measurable that
$$\mathfrak{i}_\kappa(\kappa)<2^\kappa,$$
where $\kappa>\omega$.  Furthermore, it is consistent relative to a supercompact that for some $\kappa>\omega$, $\kappa^+<2^\kappa$ and $\{ \kappa^+,2^\kappa \} \subseteq \spec(\kappa,\kappa)$.
\end{thm}

\begin{proof}
For the first claim, Kunen \cite{Kunen} showed that if $\kappa$ is measurable, then there is a forcing extension satisfying $\GCH$, preserving $\kappa$, and in which there is a normal ideal $\calI$ on $\kappa$ such that $\p(\kappa)/\calI \cong \Add(\kappa,2^\kappa)$.  By Lemma \ref{converse}, $\calI = \calI_\calF$ for some $(\kappa,\kappa)$-dmif $\calF$.  The forcing construction is essentially the same as for Theorem~\ref{indest}, but we simply force with $\Add(\alpha,2^\alpha)$ at inaccessible stages $\alpha<\kappa$. Over this model, we then force with $\bbS_\kappa^\lambda$ for any $\lambda>\kappa^+$.  A standard fusion argument shows that $\bbS_\kappa^\lambda$ preserves the property ``$\cf(\eta) > \kappa$,'' so $\kappa^+$ is preserved.  A standard $\Delta$-system argument shows that $\bbS_\kappa^\lambda$ has the $(2^\kappa)^+$-c.c.  By Lemma \ref{lem.Sacks.ind}, the maximality of $\calF$ is preserved by $\bbS_\kappa^\lambda$.  Since $\bbS_\kappa^\lambda$ forces $2^\kappa\geq\lambda$, the desired conclusion holds.

If $\kappa$ is supercompact, then by Theorem \ref{indest}, we can pass to an extension satisfying $\GCH$ at and above $\kappa$, and in which there is a $(\kappa,\kappa)$-dmif with normal density ideal.  $\bbS_\kappa^\lambda$ is $\kappa$-directed-closed for all $\lambda$, so if $\xi<(\kappa^{++})^{L(\p(\kappa))}$ is a successor and $\lambda = \kappa^{+\xi}$, then the statement ``$2^\kappa \in \spec(\kappa,\kappa)$'' holds after forcing with $\bbS_\kappa^\lambda$.  Thus in the extension, there are $(\kappa,\kappa)$-dmif's of size $\kappa^+$ and $\kappa^{+\xi}$.
\end{proof}

\section{Restrictions}\label{section.restrictions}

The methods of Section \ref{section.spectrum} to separate $\fri_\mu(\kappa)$ from $2^\kappa$ may seem a bit unsatisfying because of their reliance on the peculiarities of singular cardinal arithmetic.  The main purpose of this section is to show that this is in some sense necessary.

\begin{thm}
\label{specrest}
Suppose $\mu>\omega$ and there is a $(\mu,\kappa)$-dmif $\calF$ such that $\mu<\lambda_{\calF}$.  Then $|\calF| \geq \min\{2^{\lambda_{\calF}},\lambda_{\calF}^{+\omega} \}$.  In particular, $i_\mu(\kappa) \geq \min\{ 2^{\lambda_\mu},\lambda_\mu^{+\omega} \}$.
\end{thm}

By the proof of Theorem \ref{lowi}, the above lower bound is optimal. We will make use of the following notion:



\begin{defi} A partial order $\bbP$ is said to be \emph{$\kappa$-cofinally layered}, if there is a set of atomless regular suborders of size $<\kappa$ which is cofinal in $\p_\kappa(\bbP)$.  
\end{defi}

If $\mu<\kappa$ are such that $\alpha^{<\mu}<\kappa$ for all $\alpha<\kappa$, then for all $\theta$, $\Add(\mu,\theta)$ is $\kappa$-cofinally layered.

\begin{thm}
\label{gsgen}
Suppose $\kappa>\omega$.  If $\calI$ is a $\kappa$-complete ideal on $\kappa$ and $\p(\kappa)/\calI$ is forcing-equivalent to an atomless partial order $\bbP$ that is $\nu$-c.c.\ and $\nu$-cofinally layered for unboundedly many $\nu<\kappa$, then $| \bbP | \geq \min\{ 2^\kappa,\kappa^{+\omega}\}$.
\end{thm}


Let us first show how Theorem \ref{specrest} follows from Theorem \ref{gsgen}.  Let $\lambda = \lambda_{\calF}$, and let $\calI$ be a uniform ideal on $\kappa$ such that $\p(\kappa)/\calI \cong \Add(\mu,\theta)$, where $\theta<\min\{2^{\lambda},\lambda^{+\omega}\}$.  For all $\alpha<\lambda$, $\Add(\mu,\theta)$ is $(\alpha^{<\mu})^+$-cofinally layered, so by Lemma \ref{lem:gch}, it is $\nu$-cofinally layered for unboundedly many $\nu<\lambda$.  Also, $\theta^{<\mu} = \theta$, so $|\Add(\mu,\theta)| = \theta$.  Let $\calJ$ be the projection of $\calI$ to a $\lambda$-complete ideal on $\lambda$: $\calJ= \{ X \subseteq \lambda : 1 \Vdash_{\bbP} \lambda \notin j(X) \}$. By the claim below and Theorem \ref{gsgen} we get $\theta\geq\min\{2^\lambda,\lambda^{+\omega}\}$.
\begin{claim}
$\p(\lambda)/\calJ$ has a dense set isomorphic to poset of size $\leq \theta$ that is $\nu$-cofinally layered for unboundedly many $\nu<\lambda$.
\end{claim}
\begin{proof}Let us first show that there is a regular embedding of $\p(\lambda)/\calJ$ into $\p(\kappa)/\calI$, defined by $[X]_{\calJ} \mapsto || \lambda \in j(X) ||$.  Since this preserves antichains and $\Add(\mu,\theta)$ is $\mu^+$-c.c., $\p(\lambda)/\calJ$ is $\mu^+$-c.c.  Thus if $\{ [X_\alpha]_{\calJ} : \alpha < \delta \}$ is a maximal antichain in $\p(\lambda)/\calJ$, then $[\bigcup_\alpha X_\alpha]_{\calJ} = [\lambda]_{\calJ}$, so $\{ ||\lambda \in j(A_\alpha)|| : \alpha < \delta \}$ is maximal in $\p(\kappa)/\calI$.  Therefore, the embedding is regular.

Note that $\p(\lambda)/\calJ$ is atomless, since forcing with it necessarily adds subsets of $\mu$.
The claim now follows from a more general fact about the heredity of cofinal layering.  Suppose $\calB_0$ is a complete subalgebra of a complete Boolean algebra $\calB_1$, and $\calB_1$ has a dense subset isomorphic to a $\nu$-c.c., $\nu$-cofinally layered poset $\bbP$.  Let $\pi : \bbP \to \calB_0$ be defined by $\pi(p) = \inf \{ b \in \calB_0 : b \geq p \}$.  Then $\pi$ is a projection onto a dense subset $\bbQ$ of $\calB_0$.  To show that $\bbQ$ is $\nu$-cofinally layered, let $\tau$ be a $\bbQ$-name that is forced to be a new bounded subset of $\nu$, and let $X \in \p_\nu(\bbQ)$ a predense set of size $<\nu$ deciding the elements of $\tau$.  Let $Y \in \p_\nu(\bbP)$ be such that $\pi[Y] = X$.  Let $\bbR \supseteq Y$ be a regular suborder of $\bbP$ of size $<\nu$.  Then $\pi[\bbR]$ is a regular suborder of $\bbQ$.
\end{proof}

The proof of Theorem \ref{gsgen} is based on the proof of Theorem 2.4 of~\cite{gs2}.  We will make use of the following fact:
\begin{thm}[Shelah]
Let $\kappa<\lambda$ be regular uncountable and let $S \subseteq \lambda^+ \cap \cof(\kappa)$ be stationary.  Then there is a stationary $S' \subseteq S$ and a sequence $\lan C_\alpha : \alpha \in S' \ra$ such that each $C_\alpha$ is a club in $\alpha$ of order-type $\kappa$, and whenever $\gamma$ is a common limit point of $C_\alpha$ and $C_\beta$, then $C_\alpha \cap \gamma = C_\beta \cap \gamma$.
\end{thm}
A sequence as above is called a \emph{partial square sequence}.

\begin{proof}[Proof of Theorem \ref{gsgen}]

Suppose $\calI$ is a $\kappa$-complete ideal on $\kappa$ and $\p(\kappa)/\calI$ is forcing-equivalent to a poset $\bbP$ of size $\lambda<\min\{2^\kappa,\kappa^{+\omega} \}$ that is $\nu$-c.c.\ and $\nu$-cofinally layered for unboundedly many $\nu<\kappa$.
The case of $\lambda\leq\kappa$ is already explicitly ruled out by Theorem 1.2 of~\cite{MGSS}.  If $2^\kappa > \kappa^{+\omega}$, we can collapse $2^\kappa$ to $\lambda^+$ with $\lambda^+$-closed forcing without changing the properties of the ideal.  Thus it suffices to derive a contradiction from the assumption that $\kappa<\lambda < 2^\kappa < \kappa^{+\omega}$.  

\begin{claim}
\label{bound1}
$2^{<\kappa} \leq \lambda$.
\end{claim}

\begin{proof}
Otherwise, there is $\mu<\kappa$ and enumeration $\lan x_\alpha : \alpha< \theta \ra$ of $\p(\mu)$, where $\theta>\lambda$.  
 It is easy to show by induction that for each $n<\omega$, $\p_\kappa(\kappa^{+n})$ has a cofinal subset of size $\kappa^{+n}$.  Thus there is a collection of atomless regular suborders of $\bbP$, $\lan \bbQ_\alpha : \alpha < \lambda \ra$, which is cofinal in $\p_\kappa(\bbP)$.

Let $j : V \to M \subseteq V[G]$ be a generic ultrapower embedding, where $G \subseteq \p(\kappa)/\calI$ is generic.  Let us assume that the top condition decides the value of $j(\theta)$ as $\theta'$.  Since $\calI$ is an ideal on $\kappa$, $j[\theta]$ is cofinal in $\theta'$.  For $\alpha<\theta$, let $\zeta_\alpha$ be the minimum possible value such that some condition forces $\zeta_\alpha =\sup j[\alpha]$, and let $\eta_\alpha$ be the supremum of all possible values.  By the $\theta$-c.c., each $\eta_\alpha$ is below $\theta'$.  Note that for $\alpha<\beta$, $\zeta_\alpha<\zeta_\beta$ and $\eta_{\alpha} \leq \eta_{\beta}$.  Furthermore, both sequences $\zeta_\alpha$ and $\eta_\alpha$ are continuous.  Thus there is a club $C \subseteq \theta$ in $V$ such that for all $\alpha \in C$, it is forced that $\sup j[\alpha] = \zeta_\alpha= \eta_\alpha$.


  Let $\lan y_\alpha : \alpha < \theta' \ra = j(\lan x_\alpha : \alpha < \theta \ra)$.  For each $\alpha<j(\theta)$ there is a $\bbP$-name $\tau_\alpha$ for $y_\alpha$ of size $<\kappa$.  There is a stationary $S \subseteq \theta' \cap \cof(\kappa)$ in $V$ and a $\beta^*<\lambda$ such that each $\tau_\alpha$ is a $\bbQ_{\beta^*}$-name for $\alpha \in S$.  Let $S' = \{ \alpha \in C : \zeta_\alpha \in S\}$.  By Shelah's Theorem, there is a stationary $T \subseteq S'$ and a partial square sequence $\lan C_\alpha : \alpha \in T \ra$.  For $\xi\leq\kappa$, let $C_\alpha \restriction \xi$ denote $\{ \beta \in C_\alpha : \ot(C_\alpha \cap \beta) < \xi \}$.
  
  Let us add a generic for $\Col(\omega,|\bbQ_{\beta^*}|)$ over $V$.  In this extension, the generated ideal $\bar{\calI}$ is still $\kappa$-complete and forcing-equivalent to the same poset $\bbP$, since $\p(\kappa)^V/\calI$ is dense in $\p(\kappa)/ \bar{\calI}$.  
  If $\nu > | \bbQ_{\beta^*}|$ is such that $\bbP$ is $\nu$-c.c.\ and cofinally many $\bbQ \in \p_\nu(\bbP)$ are regular in $\bbP$, then this remains true after the collapse.
But now, $\bbQ_{\beta^*}$ is equivalent to $\Add(\omega,1)$.  For notational simplicity, let us assume this situation already occurred in $V$.  By the closure of $M$, let $\lan r_\alpha : \alpha < \kappa \ra$ be a sequence that is forced to represent the Cohen real $r$ coding the generic for $\bbQ_{\beta^*}$.  If $\beta = [\id]_G$, then $r = j(\lan r_\alpha : \alpha < \kappa \ra)(\beta)$.
  
We may assume that for each $\alpha \in T$, $\tau_{\zeta_\alpha}$ is an $\Add(\omega,1)$-name.  For $\alpha\in T$ and $\xi\leq\kappa$, let 
$$A_\alpha^\xi = \{ r_\beta : \beta<\xi \wedge (\exists \gamma \in C_\alpha \restriction \xi)\tau_{\zeta_\alpha}^{r_\beta} = x_\gamma \}.$$
For each $\alpha\in T$, $A_\alpha^\kappa$ is a nonmeager set of reals, since it is forced that $\zeta_\alpha = \sup j[\alpha] \in j(C_\alpha)$, and $\tau_{\zeta_\alpha}^r = y_{\zeta_\alpha}$.  The Cohen real $r$ is forced to be $r_\beta$ for some $\beta<j(\kappa)$, which is in $j(A_\alpha^\kappa)$.  If $A_\alpha^\kappa$ were meager in $V$, then by elementarity, a meager set coded in $V$ would cover $j(A_\alpha^\kappa)$, which therefore could not contain $r$.

Furthermore, after forcing with $\bbP$, each $A_\alpha^\kappa$ cannot even be covered by fewer than $\kappa$ meager sets.  For otherwise, in $V[G]$, we could cover $A_\alpha^\kappa$ by $\gamma$-many closed nowhere dense sets $\lan B_\beta : \beta < \gamma \ra$, where $\gamma<\kappa$.  Since each $B_\beta$ is coded by a real, there is some $\delta<\lambda$ such that each $B_\beta$ is in $V[G \cap \bbQ_\delta]$.  For each $q \in \bbQ_\delta$ and $\beta<\gamma$, let $B_{\beta,q} = \{ x \in \bbR : q \Vdash x \in \dot B_\beta \}$.  Each $B_{\beta,q}$ is a closed nowhere dense set of reals in $V$, and $\bigcup_{\beta,q} B_{\beta,q} \supseteq A_\alpha^\kappa$.  Thus by elementarity, it is forced that $j(\bigcup_{\beta,q} B_{\beta,q}) = \bigcup_{\beta,q} j(B_{\beta,q}) \supseteq j(A_\alpha^\kappa)$ (using the fact that $\gamma,|\bbQ_\delta| < \kappa$).  But the Cohen real $r$ is not in any $j(B_{\beta,q})$, since each of these is a closed nowhere dense set coded in $V$.  This contradicts the fact that $r \in j(A_\alpha^\kappa)$.

Reflecting this gives that for each $\alpha \in T$, there are unboundedly many $\xi<\kappa$ such that
 $A^\xi_\alpha$ cannot be covered by fewer than $\xi$-many meager sets.
 By the elementarity of $j$ applied to such $\xi<\kappa$, we have that for all $\alpha \in T$, it holds in $M$ that $j(A_\alpha^\xi)$ cannot be covered by $<\xi$-many meager sets.  
 The same holds in $V[G]$, by the $\kappa$-closure of $M$.
Let $\nu<\kappa$ be such that $\bbP$ is $\nu$-c.c.\ and cofinally many $\bbQ \in \p_\nu(\bbP)$ are regular in $\bbP$.  
By Fodor's Lemma, there is a stationary $T' \subseteq T$ and a $\xi^*<\kappa$ and $\rho^*<\theta$ such that $\xi^*>\nu$, and for all $\alpha \in T'$, $\sup(C_\alpha \restriction \xi^*) = \rho^*$.
Since $\lan C_\alpha : \alpha \in T \ra$ is a partial square sequence, there is a fixed set $D$ such that $D = C_\alpha \restriction \xi^*$ for all $\alpha \in T'$.

In $M$, let $\lan \sigma_\alpha : \alpha \in j(T') \ra = j(\lan \tau_{\zeta_\alpha} : \alpha\in T' \ra)$.  
Let $\alpha^* \in j(T') \setminus j[\theta]$.
There is some regular $\bbQ \subseteq \bbP$ of size $<\nu$ such that $\sigma_{\alpha^*} \in V[G \cap \bbQ]$.  
 Let $\dot\sigma$ be a $\bbQ$-name for $\sigma_{\alpha^*}$.
For each $\beta<\xi^*$, there is some $q \in \bbQ \cap G$ such that $q \Vdash \dot\sigma^{r_\beta} = f_q(r_\beta)$, where $f_q(r_\beta) = x_\gamma$ for some $\gamma \in D$.  
Each $f_q$ is a partial function in $V$, and $M \models \sigma_{\alpha^*}^s = f_{q}(s)$ for all $q \in G$ and $s \in \dom(f_{q})$.
There is some $q^* \in \bbQ \cap G$ such that $\dom(f_{q^*})$ is nonmeager in $V[G]$.  
Since $\alpha^* \notin j[T']$ and the relevant objects are smaller than $\kappa$, we can reflect twice to get some $\alpha_0<\alpha_1$ such that $\tau_{\zeta_{\alpha_0}}$,$\tau_{\zeta_{\alpha_1}}$ have the same properties as $\sigma_{\alpha^*}$.  Namely, for all $s \in \dom(f_{q^*})$, $\tau_{\zeta_{\alpha_i}}^s = f_{q^*}(s)$.
But now we reach a contradiction.  Since $\tau_{\zeta_{\alpha_0}}$,$\tau_{\zeta_{\alpha_1}}$ are $\Add(\omega,1)$-names for distinct objects, the set of reals $s$ such that $\tau_{\zeta_{\alpha_0}}^s = \tau_{\zeta_{\alpha_1}}^s$ is meager.
\end{proof}

\begin{claim}
\label{bound2}
For all $\mu<\kappa$, $\Vdash_{\p(\kappa)/\calI} j(2^\mu) < \lambda^+$.
\end{claim}

\begin{proof}
Assume $[X]_{\calI} \Vdash j(2^\mu) \geq \lambda^+$.  
Let $\lan \tau_\alpha : \alpha <\lambda^+ \ra$ be an enumeration of names of distinct elements of $j(\p(\mu))$.  Let $\lan \bbQ_\alpha : \alpha < \lambda \ra$ be a collection of regular suborders of $\bbQ$ that is cofinal in $\p_\kappa(\bbP)$.
There is a set $A \subseteq \lambda^+$ of size $\lambda^+$ and a $\beta^*<\lambda$ such that $\tau_\alpha$ is a $\bbQ_{\beta^*}$-name for $\alpha \in A$.  Each $\tau_\alpha$ is coded by a subset of $\nu=\max\{\mu,|\bbQ_{\beta^*}|\}$.  But this means that $2^\nu > \lambda$, contrary to Claim \ref{bound1}.
\end{proof}

To finish the proof of Theorem \ref{gsgen}, let $\theta = 2^\kappa$.  Since $\p(\kappa)/I$ adds bounded subsets of $\kappa$, there is some $\mu<\kappa$ such that $2^\mu \geq \kappa$.  By Claim \ref{bound2}, it is forced that $j(\kappa) <\lambda^+$.  Since all cardinals above $\kappa$ are preserved, and there are only finitely many cardinals of $M$ between $j(\kappa)$ and $j(\theta)$, and $j$ is continuous at regular cardinals above $\kappa$, we must have that $j(\eta) = \eta$ for cardinals $\eta$, $\lambda^+\leq\eta\leq\theta$.  Since $\p(\kappa)^V \subseteq M$, it must be the case that $(2^\kappa)^M \geq \theta$.  Therefore,
$$M \models \theta \leq 2^\kappa \leq 2^{j(\kappa)} = j(2^\kappa) = \theta.$$
Thus in $M$, there is some $\mu<j(\kappa)$ such that $2^\mu = 2^{j(\kappa)}$.  Reflecting this to $V$ and using Claim \ref{bound1}, we get $\lambda = 2^{<\kappa} = 2^\kappa.$  This contradiction completes the proof.
\end{proof}

Theorem \ref{gsgen} fails to generalize to uniform $\kappa$-complete ideals on $\theta$, where $\theta>\kappa$.  For example, suppose $\kappa$ is a huge cardinal.  This is witnessed by a normal $\kappa$-complete ultrafilter on $[\theta]^\kappa$ for some inaccessible $\theta$.  If $j : V \to M$ is the associated embedding, then $j(\kappa) = \theta$.  Let $\mu<\kappa$ be regular, and force with $\Add(\mu,\kappa)$.  By Theorem \ref{foreman}, in the extension, there is a uniform $\kappa$-complete ideal $\calI$ on $\theta$ such that $\p(\theta)/\calI$ has a dense set isomorphic to $\Add(\mu,\theta)$.  Of course, $\theta<2^\theta$, and we may assume that $\theta^{+\omega}$ is a strong limit.

This example shows that Lemma \ref{converse} is only a partial converse to Lemma \ref{lem:ideal}.  By Proposition \ref{ineq2}, there is no $(\mu,\theta)$-mif of size $\theta$, assuming that $\theta^{<\mu} = \theta$.  The explanation is that, in Lemma \ref{converse}, we made essential use of the assumption that the number of Cohen sets added by forcing with the ideal is at least the size of a cofinal subset of the ideal.

We can also bring this situation down to uniform $\kappa$-complete ideals on $\kappa^{+n}$ for finite $n$.  The idea for the construction comes from Magidor~\cite{MM}.  Let us say that a cardinal $\kappa$ is \emph{$+n$-huge} if there is an elementary embedding $j : V \to M$ such that $M$ is $j(\kappa)^{+n}$-closed.  This property is witnessed by a normal $\kappa$-complete ultrafilter on the set $\{ z \subseteq \theta^{+n} : (\forall m \leq n) \ot(z \cap \theta^{+m}) = \kappa^{+m} \}$, where $j(\kappa) = \theta$.

\begin{prop}
Assume $\ZFC$ is consistent with a 2-huge cardinal, and let $n<m$ and $k$ be natural numbers.  Then there is a model of $\ZFC$ in which for some regular cardinal $\lambda>\aleph_k=2^{<\aleph_k}$, there is a uniform $\lambda$-complete ideal $\calI$ on $\lambda^{+n+m+2}$ such that $\p(\lambda^{+n+m+2})/\calI$ has a dense set isomorphic to $\Add(\aleph_k,\lambda^{+m+2})$. 
\end{prop}

\begin{proof}
We can assume that there is a model of $\ZFC+\GCH$ with a cardinal $\kappa$ that is $+n$-huge, and a supercompact $\lambda<\kappa$.  Let $j : V \to M$ be an embedding witnessing $\kappa$ is $+n$-huge with $j(\kappa)=\theta$.  Assume $j$ is derived from a normal ultrafilter $\calU$ on the set $Z \subseteq \p(\theta^{+n})$ as above, and note that $|Z|=\theta^{+n}$.  Use Laver's forcing to make $\lambda$ indestructibly supercompact.  Since this forcing has size $\lambda$, $\kappa$ remains $+n$-huge with the same target.

By Kunen~\cite{Kunen2}, there is a $\kappa$-c.c., $\lambda$-directed-closed forcing $\bbP \subseteq V_\kappa$ that turns $\kappa$ into $\lambda^+$, and has the property that there is a regular embedding from $\bbP * \dot{\Col}(\kappa^{+m},{<}\theta)$ into $j(\bbP)$ that is the identity on $\bbP$.  Let $G * H \subseteq \bbP * \dot{\Col}(\kappa^{+m},{<}\theta)$ be generic.  Force further to obtain $G' \subseteq j(\bbP)$ containing the image of $G * H$.  Since $j(\bbP)$ is $\theta$-c.c., $M[G']$ is $\theta^{+n}$-closed in $V[G']$.  $j[H]$ has a lower bound $q \in\Col(\theta^{+m},{<}j(\theta))^{M[G']}$.  Using $\GCH$, we find that $|j(\theta)| = \theta^{+n+1}$.  Using this and the closure of $M[G']$, we can work in $V[G']$ to construct a filter $H'$ containing $q$ that is $\Col(\theta^{+m},{<}j(\theta))$-generic over $M[G']$, allowing an extension of the embedding to $j : V[G*H] \to M[G' * H']$.  As in the argument for Theorem \ref{indest}, there is in $V[G*H]$ a normal ideal $\calI$ on $Z$ extending the dual of $\calU$ such that $\p(Z)/\calI$ is forcing-equivalent to the quotient $j(\bbP)/(G*H)$.

$\p(Z)/\calI$ is a complete Boolean algebra that is $\theta$-c.c.\ and of size $\theta$.  We claim that there are at most $\theta$-many functions $f : Z \to \lambda$ that are distinct modulo $\calI$.  Since $\calI$ is $\lambda^+$-complete, the equivalence class of such a function is determined by a partition of $Z$ into $\leq\lambda$-many $\calI$-positive pieces.  This corresponds to a maximal antichain of size $\leq\lambda$ in $\p(Z)/\calI$.  Since $\theta^\lambda = \theta$, there are only $\theta$-many possibilities.  
Furthermore, there are at most $\theta$-many functions $f : Z \to \kappa$ that are distinct modulo $\calI$.  To show this, for each $\alpha<\theta$, let $f_\alpha : Z \to \kappa$ be such that $f_\alpha(z) = \beta$, where $\alpha$ is the $\beta^{th}$ element of $z$.  (By normality, almost all $z$ have $\alpha \in z \cap \theta$ and $\ot(z \cap \theta)=\kappa$.)  If $g : A \to \kappa$ is any function defined on an $\calI$-positive set $A$, then by normality, $g$ is dominated by some $f_\alpha$ on an $\calI$-positive $B \subseteq A$.  By the $\theta$-c.c., every $g : Z \to \kappa$ is dominated by some $f_\alpha$ modulo $\calI$.  Since the set of predecessors of any $f_\alpha$ has cardinality at most $\theta$, there are in total at most $\theta$-many $g : Z \to \kappa$ that are distinct modulo $\calI$.  Finally, note that by a similar argument, there is a mod-$\calI$-increasing sequence of length $\theta^+$ of functions from $Z$ to $\kappa^+$.

Now, $\lambda$ is still strongly compact in $V' = V[G*H]$.  Thus we can extend the dual of $\calI$ to a $\lambda$-complete ultrafilter $\calW$ on $Z$.  If $i : V' \to N$ is the ultrapower embedding by $\calW$, then $\crt(i)=\lambda$, and there are at most $\theta$-many functions from $Z$ to $\kappa$ that are distinct modulo $\calW$.  Thus $i(\kappa) < \theta^+$.  Since there are, modulo $\calW$, $\theta^+$-many distinct functions from $Z$ to $\kappa^+$,  $i(\kappa)\geq\theta$.
 If we then force with $\Add(\aleph_k,\kappa)$, then Theorem \ref{foreman} implies that in the extension, there is a $\lambda$-complete uniform ideal $\calJ$ on $Z$ such that $\p(Z)/\calJ$ has a dense set isomorphic to $\Add(\aleph_k,i(\kappa))$.  Finally, we note that in this model, $\theta=\lambda^{+m+2}$ and $|Z| = \lambda^{+n+m+2}$.
\end{proof}

\section{Questions}~\label{section.questions}

The following questions remain of interest.

\begin{enumerate}
\item Is it consistent that for some regular $\kappa>\omega$, $\kappa^+<\fri_\kappa(\kappa) < 2^\kappa$?
\item Given uncountable cardinals $\mu\leq\kappa$, determine the possibilities for $\spec(\mu,\kappa)$.  Is it consistent that such a set is infinite?
\item Is it consistent that for some $\kappa$, $\spec(\mu,\kappa) \not= \emptyset$ for infinitely many $\mu$?
\item Corollary \ref{separation} says that if $\mu_0<\mu_1$ are such that there exist both maximal $\mu_0$-free and maximal $\mu_1$-free families, then $\mu_0\leq\lambda_{\mu_0}\leq\mu_1\leq\lambda_{\mu_1}$.  One of the first two inequalities must be strict.  Is it possible that $\lambda_{\mu_0} = \mu_1$?
\end{enumerate}


\begin{thebibliography}{00}

\bibitem{BBPSPV}
B. Balcar, P. Simon, P. Vojt\'a\v{s}
{\emph{Refinement properties and extensions of filters in boolean
    algebras}}
Transactions of the American Mathematical Society, Vol. 267, N. 1,
1981, 265-283.



\bibitem{VFDM1}
V. Fischer, D. Montoya
{\emph{Ideals of independence}}, 
Archive for Mathematical Logic 58(5-6), 767-785, 2019.

\bibitem{VFDM2}
V. Fischer, D. C. Montoya
{\emph{Higher Independence}}
preprint, 2020.

\bibitem{VFSS1}
    V. Fischer, S. Shelah
   {\emph{The spectrum of independence}}
      Archive for Mathematical Logic 58(7-8), 877-884, 2019.

 
      
\bibitem{Foreman}
M. Foreman
{\emph{Calculating quotient algebras of generic embeddings}} 
Israel J. Math. 193 (2013), no. 1, 309–341. 


\bibitem{MGSS} 
M. Gitik, S. Shelah
{\emph{Forcing with ideals and simple forcing notions}}
Israel Journal of Mathematics,  Vol. 68, No. 2, 1989.

\bibitem{gs2}
M. Gitik, S. Shelah
{\emph{More on simple forcing notions and forcings with ideals}}
Fourth Asian Logic Conference (Tokyo, 1990). 
Ann. Pure Appl. Logic 59 (1993), no. 3, 219–238. 

\bibitem{Hamkins}
J. D. Hamkins
{\emph{Extensions with the approximation and cover properties have no new large cardinals}},
Fund. Math., 180 (3), 2003, 257-277.
	
\bibitem{TJ}
  T. Jech
  {\emph{Set Theory}}
  The third millennium edition, revised and expanded. Springer Monographs in Mathematics. Springer-Verlag, Berlin, 2003. xiv+769 pp.


\bibitem{TJKP}
  T. Jech, K. Prikry
  {\emph{On ideals of sets and the power set operation}}
  Bull. Amer. Math. Soc. 82 (1976), no. 4, 593-595.

\bibitem{Kanamori}
   A. Kanamori
   {\emph{Perfect-set forcing for uncountable cardinals}}
   Annals of Mathematical Logic 19 (1980), 97-114.

\bibitem{Kunen}
K. Kunen
{\emph{Maximal $\sigma$-independent families}}
Fundamenta Mathematicae,  117 (1), 75 - 80 (1983).

\bibitem{Kunen2}
K. Kunen
{\emph{Saturated ideals}}
J. Symbolic Logic, 43 (1), 1978, 65-76.

\bibitem{MM}
 M. Magidor 
{\emph{On the existence of nonregular ultrafilters and the cardinality
    of ultrapowers}}
Transactions of the American Mathematical Society, 249(1), 1979, 97-111.

\bibitem{Unger}
S. Unger
{\emph{Fragility and indestructibility {II}}},
Ann. of Pure and Appl. Logic, 166 (11), 2015, 1110-1122


\bibitem{Usuba}
T. Usuba
{\emph{The approximation property and the chain condition}},
Research Institute for Mathematical Sciences Kokyuroku, 1895, 2014, 103-107.

\bibitem{Velleman}
D. Velleman
{\emph{Morasses, diamond, and forcing}}, Ann. Math. Logic, 23 (2-3), 1982, 199-281.




\end{thebibliography}
\end{document}